\documentclass[a4paper, 10pt, DIV10, headinclude=false, footinclude=false]{scrartcl}

\usepackage{amsthm, amsmath, amssymb, amsfonts}
\usepackage[utf8]{inputenc}
\usepackage[english]{babel}
\usepackage{url}
\usepackage{mathtools}

\usepackage{tikz}
\usepackage{pgfplots}
\usepackage{pgfplotstable}

\usepackage[caption=false]{subfig}
\usepackage{booktabs}

\newtheorem{theorem}{Theorem}[section]
\newtheorem{lemma}[theorem]{Lemma}
\newtheorem{corollary}[theorem]{Corollary}
\newtheorem{proposition}[theorem]{Proposition}
\newtheorem{assumption}[theorem]{Assumption}

\theoremstyle{definition}
\newtheorem{definition}[theorem]{Definition}

\numberwithin{figure}{section}
\numberwithin{table}{section}
\numberwithin{equation}{section}
\numberwithin{theorem}{section}

\newenvironment{abstr}[1]{ \vspace{.05in}\footnotesize
       \parindent .2in
         {\upshape\bfseries #1. }\ignorespaces}{\par\vspace{.1in}}
\newenvironment{Abstract}{\begin{abstr}{Abstract}}{\end{abstr}}
\newenvironment{keywords}{\begin{abstr}{Key words}}{\end{abstr}}
\newenvironment{AMS}{\begin{abstr}{AMS subject classifications}}{\end{abstr}}

\DeclareMathOperator{\diam}{diam}
\DeclareMathOperator{\eff}{eff}
\DeclareMathOperator{\Id}{Id}
\DeclareMathOperator{\stab}{stab}
\DeclareMathOperator{\reg}{reg}
\DeclareMathOperator{\appr}{appr}
\DeclareMathOperator{\HMM}{HMM}
\renewcommand{\Re}{\operatorname{Re}}
\renewcommand{\Im}{\operatorname{Im}}
\newcommand{\halb}{\frac 12}
\DeclareMathOperator*{\wto}{\rightharpoonup}
\newcommand{\twosc}{\stackrel{2}{\wto}}
\newcommand{\nz}{\mathbb{N}}       
\newcommand{\gz}{\mathbb{Z}}       
\newcommand{\rz}{\mathbb{R}}       
\newcommand{\cz}{\mathbb{C}}       
\newcommand{\pz}{\mathbb{P}}
\newcommand{\de}{\delta}

\newcommand{\ep}{\varepsilon}

\newcommand{\om}{\omega}
\newcommand{\Om}{\Omega}

\newcommand{\si}{\sigma}

\newcommand\Ve{\mathbf{e}}

\newcommand\Vv{\mathbf{v}}
\newcommand\Vu{\mathbf{u}}

\newcommand\Vz{\mathbf{z}}

\newcommand\VH{\mathbf{H}}

\newcommand\VV{\mathbf{V}}
\newcommand\Vpsi{\boldsymbol{\psi}}

\newcommand\CB{\mathcal{B}}

\newcommand\CH{\mathcal{H}}

\newcommand\CT{\mathcal{T}}

\usepackage[textsize=footnotesize]{todonotes}
\usepackage[xcolor=dvipdf]{changes}
\setremarkmarkup{\todo[color=Changes@Color#1!40]{\textcolor{black}{#1:{#2}}}}

\allowdisplaybreaks[4]

\title{A new Heterogeneous Multiscale Method for the Helmholtz equation with high contrast%
\thanks{This work was supported by the Deutsche Forschungsgemeinschaft (DFG) in the project ``OH 98/6-1: Wellenausbreitung in periodischen Strukturen und Mechanismen negativer Brechung''} 
}

\author{Mario Ohlberger%
\thanks{Angewandte Mathematik: Institut f\"ur Analysis und Numerik, Westf\"alische Wilhelms-Uni\-ver\-si\-t\"at M\"unster, D-48149 M\"unster
}
\and Barbara Verf\"urth\footnotemark[2]}

\date{}

\begin{document}

\maketitle

\begin{Abstract}
In this paper, we suggest a new Heterogeneous Multiscale Method (HMM) for the Helmholtz equation with high contrast. 
The method is constructed for a setting as in Bouchitt{\'e} and Felbacq ({\itshape C.R.\ Math.\ Acad.\ Sci.\ Paris} 339(5):377--382, 2004), where the high contrast in the parameter leads to unusual effective parameters in the homogenized equation. 
We revisit existing homogenization approaches for this special setting and analyze the stability of the two-scale solution with respect to the wavenumber and the data. 
This includes a new stability result for solutions to the Helmholtz equation with discontinuous diffusion matrix.
The HMM is defined as direct discretization of the two-scale limit equation.
 With this approach we are able to show quasi-optimality and an a priori error estimate under a resolution condition that inherits its dependence on the wavenumber from the stability constant for the analytical problem. 
Numerical experiments confirm our theoretical convergence results and examine the resolution condition.
Moreover, the numerical simulation gives a good insight and explanation of the physical phenomenon of frequency band gaps.
\end{Abstract}



\section{Introduction}
The interest in (locally) periodic media, such as photonic crystals, has grown in the last years as they exhibit astonishing properties such as band gaps or negative refraction, see \cite{EP04negphC, PE03lefthanded, CJJP02negrefraction}. 
In this paper, we study {\itshape artificial magnetism} in the setting of \cite{BF04homhelmholtz}, which has been inspired by the experimental set-up of \cite{OBP02magneticactivity}.

The electro-magnetic properties of a material are governed by the permittivity $\ep$ and the permeability $\mu$. Whereas  for $\ep$ a great range of values can be observed, almost all materials are non-magnetic, i.e.\ $\mu$ is close to $1$. 
Artificial magnetism now describes the occurrence of an (effective) permeability $\mu_{\eff}\neq 1$ in an originally non-magnetic material with $\mu=1$.
Clearly, such a material must exhibit some interior structure to allow this significant change of behavior. 
In \cite{BF04homhelmholtz}, an unusual and highly heterogeneous scaling (in the sense of Allaire \cite[Section 4]{All92twosc}) of material parameters (see below) has been used to obtain a frequency-dependent permeability, which can even have a negative real part, in the homogenization limit.
The observation that $\mu_{\eff}$ can even be negative is of particular interest: When $\ep$ and $\mu$ are negative, such a material can have a negative refraction index, as discussed in \cite{Ves}. 
Metals can have a negative real part of $\ep$, but no negative $\mu$ can be observed in nature.
Moreover, in material with positive $\varepsilon$ and negative $\mu$, wave propagation is forbidden, which corresponds to a frequency in the band gap.

The setting of \cite{BF04homhelmholtz}, inspired by \cite{OBP02magneticactivity} and \cite{BF97homfibres}, is the following (see also Figure \ref{fig:setting}):
A periodic array of rods with high permittivity (depicted in gray in Figure \ref{fig:setting}) is embedded in a lossless dielectric material.
Denoting by the small parameter $\de$ the periodicity, the high permittivity in the rods is modeled by setting $\ep^{-1}=\de^2\ep_i^{-1}$, see Figure \ref{sec:setting} for an exact definition.
The consideration of small inclusions with high permittivity has become a popular modeling also in the three-dimensional setting to tune unusual effective material properties, see \cite{BBF09hom3d, BS10splitring, BS13plasmonwaves, CC15hommaxwell, LS15negindex}.

The overall setting in this paper can be described now as follows: We consider a scatterer  of the form $\Om\times \rz$ with $\Om\subset\rz^2$ bounded and smooth (with $C^2$ boundary). The structure is non-magnetic, i.e.\ $\mu=1$, and has a relative permittivity $\ep_r$, which equals $1$ outside $\Om$.
This effectively two-dimensional geometry (invariant in $x_3$-direction) is illuminated by a transversally polarized field $\VH_{inc}=(0, 0, u_{inc})^T$.
The total magnetic field $\VH=(0, 0, u)^T$ then satisfies the Helmholtz equation
\begin{equation}
\label{eq:helmholtz}
-\nabla\cdot (\ep_r^{-1}\nabla u)-k^2 u=0 \quad \text{on}\quad \rz^2
\end{equation}
with the wave number $k=\om/c$.
We artificially truncate our domain by introducing a sufficiently large convex Lipschitz domain $G\supset\supset \Omega$ and imposing on $\partial G$ the following boundary condition
\begin{equation}
\label{eq:sommerfeldapprox}
\nabla u\cdot n-iku=g:=\nabla u_{inc}\cdot n-iku_{inc},
\end{equation}
which is the popular first order approximation of the Sommerfeld radiation condition, cf.\ \cite{DS13helmholtzaposteriori, Ihl98}.
The relative permittivity $\ep_r=a_\de^{-1}$ inside the scatterer models the described setting of periodic inclusions with high permittivity and is defined in \eqref{eq:hetparam}.
Throughout this article, we assume that there is $k_0>0$ such that $k\geq k_0$, which corresponds to medium and high frequencies.

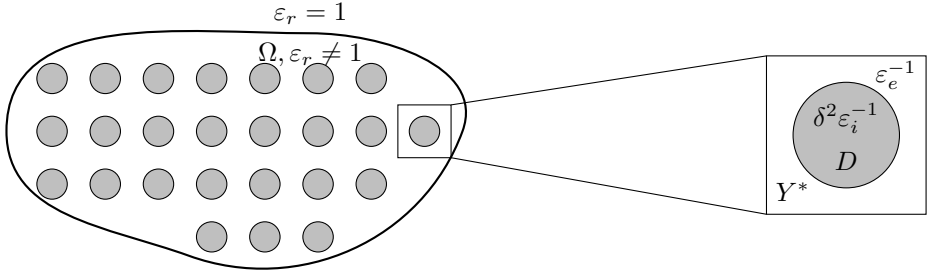
\begin{figure}
\begin{tikzpicture}
\draw[thick] (0,0) to [out=90, in=180] (4,1.3)
to [out=0, in=70] (6,0)
to [out=250, in=340] (2.4,-1.65)
to [out=160, in=270](0,0);
\node [above] at(4,1.3){$\ep_r=1$};
\node [below] at(4,1.3){$\Om, \ep_r\neq1$};
\draw[fill=lightgray] (0.6,0) circle [radius=0.2];
\draw[fill=lightgray] (1.3,0) circle [radius=0.2];
\draw[fill=lightgray] (2.0,0) circle [radius=0.2];
\draw[fill=lightgray] (2.7,0) circle [radius=0.2];
\draw[fill=lightgray] (3.4,0) circle [radius=0.2];
\draw[fill=lightgray] (4.1,0) circle [radius=0.2];
\draw[fill=lightgray] (4.8,0) circle [radius=0.2];
\draw[fill=lightgray] (5.5,0) circle [radius=0.2];
\draw[fill=lightgray] (0.6,0.7) circle [radius=0.2];
\draw[fill=lightgray] (1.3,0.7) circle [radius=0.2];
\draw[fill=lightgray] (2.0,0.7) circle [radius=0.2];
\draw[fill=lightgray] (2.7,0.7) circle [radius=0.2];
\draw[fill=lightgray] (3.4,0.7) circle [radius=0.2];
\draw[fill=lightgray] (4.1,0.7) circle [radius=0.2];
\draw[fill=lightgray] (4.8,0.7) circle [radius=0.2];
\draw[fill=lightgray] (0.6,-0.7) circle [radius=0.2];
\draw[fill=lightgray] (1.3,-0.7) circle [radius=0.2];
\draw[fill=lightgray] (2.0,-0.7) circle [radius=0.2];
\draw[fill=lightgray] (2.7,-0.7) circle [radius=0.2];
\draw[fill=lightgray] (3.4,-0.7) circle [radius=0.2];
\draw[fill=lightgray] (4.1,-0.7) circle [radius=0.2];
\draw[fill=lightgray] (4.8,-0.7) circle [radius=0.2];
\draw[fill=lightgray] (2.7,-1.4)circle [radius=0.2];
\draw[fill=lightgray] (3.4,-1.4)circle [radius=0.2];
\draw[fill=lightgray] (4.1,-1.4)circle [radius=0.2];
\draw (5.15, -0.35) rectangle (5.85,0.35);
\draw (5.85, 0.35) -- (10,1);
\draw (5.85,-0.35) -- (10,-1.1);
\draw (10,-1.1) rectangle (12.1,1);
\draw[fill=lightgray] (11.05, -0.05) circle[radius=0.7];
\node [below] at(11.05,-0.15) {$D$};
\node [above] at (11.05, -0.15) {$\de^2\ep_i^{-1}$};
\node [right] at (10,-0.8) {$Y^*$};
\node [left] at (12.1, 0.75) {$\ep_e^{-1}$};
\end{tikzpicture}
\caption{Left: Scatterer $\Om$ with highly conductive inclusions $D_\de$ (in gray); Right: Zoom into one unit cell $Y$ and scaling of the permittivity $\ep_r^{-1}$.}
\label{fig:setting}
\end{figure}

A numerical treatment of \eqref{eq:helmholtz} with boundary condition  \eqref{eq:sommerfeldapprox} and permittivity with high contrast is very challenging. Solutions to Helmholtz problems show oscillatory behavior in general and the consideration of (locally) periodic media intensifies this effect. 
The challenge is then to well approximate the heterogeneities in the material and the oscillations induced by the incoming wave.
It is important to relate the scales of these oscillations: We basically have a three-scale structure here with $\delta\ll k^{-1}<1$, i.e.\ the periodicity of the material (and the size of the inclusions) is much smaller than the wavelength of the incoming wave.
A direct discretization requires a grid with mesh size $h<\de\ll 1$ to approximate the solution faithfully. This can easily exceed today's computational resources when using a standard approach. In order to make a numerical simulation feasible, so called multiscale methods can be applied.
 The family of Heterogeneous Multiscale Methods (HMM) \cite{EE03hmm, EE05hmm} is a class of multiscale methods that has been proved to be very efficient for scale-separated locally periodic problems. The HMM can exploit local periodicity in the coefficients to solve local sample problems that allow to extract effective macroscopic features and to approximate solutions with a complexity independent of the (small) periodicity $\de$. 
First analytical results concerning the approximation properties of the HMM for elliptic problems have been derived in \cite{Abd05aprioriHMM, EMZ05analysisHMM, Glo06, Ohl05HMM} and then extended to other problems, such as time-harmonic Maxwell's equations \cite{HOV15maxwellHMM}.
Other related works are the HMM for Helmholtz problems with locally periodic media (without high contrast!) \cite{CS14hmmhelmholtz}, or a multiscale asymptotic expansion for the Helmholtz equation \cite{CCZ02asymphelmholtz}.

The new contribution of this article is the first formulation of a Heterogeneous Multiscale Method for the Helmholtz equation with high contrast in the setting of \cite{BF04homhelmholtz}, its comprehensive numerical analysis and its implementation.
The numerical experiment not only shows the practicability of the suggested HMM, but also gives an enlightening insight into the physical background of artificial magnetism and frequency band gaps. 
The HMM can be used to approximate the true solution to \eqref{eq:helmholtz} with a much coarser mesh and hence less computational effort.
We observe that for a frequency in the band gap, wave propagation is prohibited due to destructive interference of waves incited at eigen resonances of the small and highly permittive inclusions.
From the theoretical point of view, the main result is that the energy error converges with rate $k^{q+1}(H+h)$ if the resolution condition $k^{q+2}(H+h)=O(1)$ is fulfilled. 
Here, $H$ and $h$ denote the $\delta$-independent mesh sizes used for the HMM and we assume that the analytical two-scale solution has a stability constant of order $k^q$ with $q\in \nz_0$.
This resolution condition is unavoidable for standard Galerkin discretizations of Helmholtz problems and it shows up with $q=0$ (the optimal case) in our numerical experiments.
A posteriori estimates in this setting are equally possible to obtain. The described HMM itself might be transferable/adaptable to similarly scaled situations in three dimensions.

To complement our numerical analysis, we also show an explicit stability estimate for the solution to the two-scale limit equation, so that we have an explicit (though maybe sub-optimal) result for the stability exponent: $q=3$.
This includes  a second contribution, which may be of own interest: a new stability result for a certain class of Helmholtz-type problems, namely with  matrix-valued discontinuous diffusion coefficient. 
Stability results for the Helmholtz equation have only been proved in the following cases: 
Constant coefficients have been studied under various geometrical conditions in \cite{BSW16helmholtz,EM12helmholtz, Het07helmholtz, Mel95gFEM, MS11helmholtz, MS14signindefinite} and scalar-valued, globally Lipschitz continuous coefficients have been treated in \cite{BGP15hethelmholtzLOD}.

The article is organized as follows: In Section \ref{sec:setting} we detail the (geometric) setting of the heterogeneous problem considered and give some basic notation used throughout the article. 
We present and combine existing homogenization results and analyze the homogenized problems in detail in Section \ref{sec:homogenization}. 
This is the motivation and starting point for the formulation of the corresponding HMM in Section \ref{sec:HMM}. 
The quasi-optimality and a priori estimates for the new method as the central statement of the article are given in Section \ref{sec:quasiopt}. All essential proofs are detailed in Section \ref{sec:proofs}.
A numerical experiment is presented in Section \ref{sec:numexperiment}.

\section{Problem setting}
\label{sec:setting}
For the remainder of this article, let $\Om\subset\subset G\subset \rz^2$ be two bounded domains, where $\partial \Om$ is of class $C^2$ and $G$ is convex and has a polygonal Lipschitz boundary.
Throughout this paper, we use standard notation: 
For a domain $\om$, $p\in[1,\infty)$ and $s\in \rz_{\geq 0}$, $L^p(\om)$ denotes the usual complex Lebesgue space with norm $\|\cdot \|_{L^p(\om)}$ and $H^s(\om)$ denotes the complex (fractional) Sobolev space with the norm $\|\cdot \|_{H^s(\om)}$. The domain $\om$ is omitted from the norms if no confusion can arise.
The dot will denote a normal (real) scalar product, for a complex scalar product we will explicitly conjugate the second component by using $v^*$ as the conjugate complex of $v$.
The $L^2$ scalar product on a domain $\om$ is abbreviated by $(\cdot, \cdot)_\om$ and the corresponding norm abbreviated by $\|\cdot\|_\om$.
For a polygonally bounded domain $\om$, $H^{1/2}(\partial \om)$ denotes the space of functions which are edge-wise $H^{1/2}$. For the domain $G$, we abbreviate by $H^{s}_{pw}(G):=H^s(\Om)\cap H^s(G\setminus\overline{\Om})\cap H^1(G)$ for $s>1$ the function space of piece-wise $H^s$ functions and note that $H^s_{pw}(G)=H^s(G)$ for $s\in [1, \frac{3}{2})$, see \cite{Petz10reginterface}.
For $v\in H^1(\om)$, we frequently use the $k$-dependent norm
\[\|v\|_{1,k,\om}:=\bigl(\|\nabla v\|^2_\om+k^2\|v\|^2_\om\bigr)^{1/2},\]
which is obviously equivalent to the $H^1$-norm.

Let $\Ve_j$ denote the $j$'th unit vector in $\rz^2$. For the rest of the paper, we write $Y:=[-\halb, \halb)^2$ to denote the $2$-dimensional unit square and we say that a function $v\in L^2_{loc}(\rz^2)$ is $Y$-periodic if it fulfills $v(y)=v(y+\Ve_j)$ for all $j=1,2$ and almost every $y\in\rz^2$.
With that we denote $L^2_\sharp(Y):=\{v\in L^2_{loc}(\rz^2)| v\text{ is }Y\text{-periodic}\}$.
Analogously we indicate periodic function spaces by the subscript $\sharp$.
 For example, $H^1_\sharp(Y)$ is the space of periodic $H^1_{loc}(\rz^2)$ functions and we define
\[H^1_{\sharp,0}(Y):=\left\{\phi\in H^1_\sharp(Y)\Bigl|\int_Y\phi=0\right\}.\]
For $Y^*\subset Y$, we denote by $H^1_{\sharp,0}(Y^*)$ the restriction of functions in $H^1_{\sharp,0}(Y)$ to $Y^*$. For $D\subset\subset Y$, $H^1_0(D)$ can be interpreted as subspace of $H^1_\sharp(Y)$ and we will write $H^1_0(D)_\sharp$ to emphasize this periodic extension.
By $L^p(\Om;X)$ we denote Bochner-Lebesgue spaces over the Banach space $X$ and we use the short notation $f(x,y):=f(x)(y)$ for $f\in L^p(\Om; X)$. Functions in $L^2(\Om)$ are also regarded as functions in $L^2(G)$ by simple extension by zero. 

Using the above notation we consider the following setting for the (inverse) relative permittivity $\ep_r^{-1}$, see \cite{BF04homhelmholtz}:
 $\Om$ is composed of $\de$-periodically disposed sections of rods, $\de$ being a small parameter. Denoting by $D\subset\subset Y$ a connected domain with $C^2$ boundary, the rods occupy a region $D_\de:=\cup_{j\in I} \de(j+D)$ with $I=\{j\in \gz^2|\delta(j+Y)\subset \Omega\}$. The complement of $D$ in $Y$, which is also connected, is denoted by $Y^*$. The inverse relative permittivity $a_\de:=\ep_r^{-1}$ is then defined (possibly after rescaling) as (cf.\ Figure \ref{fig:setting})
\begin{align}
\label{eq:hetparam}
a_\de(x):= 
\begin{cases}
\de^2\ep_i^{-1} & \text{if } x\in D_\de\qquad\quad\;\; \text{ with }\ep_i\in \cz, \Im(\ep_i)>0, \Re(\ep_i)>0,\\
\ep_e^{-1} & \text{if }x\in \Om\setminus D_\de\qquad \text{ with }\ep_e\in\rz_+,\\
1 & \text{if }x \in G\setminus\overline{\Omega}.
\end{cases}
\end{align}
We assume $\Re(\ep_i)>0$ for simplicity; all results hold -- up to minor modifications in the proofs -- also for $\ep_i$ with $\Re(\ep_i)\leq 0$. 
Physically speaking, this means that the scatterer $\Om$ consists of periodically disposed metallic rods $D_\de$ embedded in a dielectric ``matrix'' medium. 
The scaling of $\de^2$ in the rods corresponds to a constant optical diameter of these inclusions.

It is essential that $\Om\setminus D_\de$ is connected, otherwise the two-scale convergences shown below can fail, see \cite{BS13plasmonwaves} for an example. To assume $D$ as connected is only done for simplicity.

\begin{definition}[Weak solution]
Let the parameter $a_\de$ be defined by \eqref{eq:hetparam} and let $g\in H^{1/2}(\partial G)$. 
We call $u_\de\in H^1(G)$ a weak solution if it fulfills
\begin{equation}
\label{eq:weakHelmholtz}
\int_G a_\de(x)\nabla u_\de\cdot \nabla\psi^*-k^2u_\de\psi^*\, dx-ik\int_{\partial G}u_\de\psi^*\, d\si=\int_{\partial G}g\psi^*\, d\si\quad \forall \psi\in H^1(G).
\end{equation}
\end{definition}
It is well known that for fixed $\de$, there is a unique solution to \eqref{eq:weakHelmholtz}, which can be seen using the Fredholm alternative: The left-hand side fulfills a G{\aa}rding inequality and problem \eqref{eq:weakHelmholtz} as well as the adjoint problem are uniquely solvable.
Throughout the article, $C$ denotes a generic constant, which does not depend on $k$ (and later the mesh sizes $H$ and $h$), but may depend on $k_0$ and may vary from line to line.

\section{Homogenization and analysis of the homogenized equations}
\label{sec:homogenization}
As the parameter $\de$ is assumed to be very small in comparison to the wavelength and the typical length scale of $\Om$, one can reduce the complexity of problem \eqref{eq:weakHelmholtz} by considering the limit $\de\to 0$. This process, called homogenization, can be performed with the tool of two-scale convergence \cite{All92twosc, LNW02twosc} for locally periodic problems. 
In Subsection \ref{subsec:twoscale}, we adopt the two-scale equation from \cite[Section 4]{All92twosc}, derived for highly heterogeneous diffusion problems with Dirichlet boundary condition, and the homogenized effective macroscopic equation from \cite{BF04homhelmholtz} (with Sommerfeld radiation condition) to our setting.
Subsection \ref{subsec:analysis} is devoted to a detailed analysis of the two-scale equation and its homogenized formulation.
Most importantly, this subsection includes a new stability result for solutions to Helmholtz-type problems, generalizing results available in the literature to a larger class of coefficients.
We emphasize that this analysis is an important building block and prerequisite for the numerical analysis in Section \ref{sec:HMM}.

\subsection{Two-scale equation and homogenized formulation}
\label{subsec:twoscale}

Two-scale convergence is a special form of convergence for locally periodic functions, which tries to capture oscillations and lies between weak and strong (norm) convergence. Its definition and main properties can be found in \cite{All92twosc} or \cite{LNW02twosc}, for instance. We write $\twosc$ for the two-scale convergence in short form. 

The special scaling of $a_\de$ with $\de^2$ on a part of $\Om$ leads to a different behaviour of the solution on $D_\de$ and its complement, which can still be seen in the two-scale equation and the homogenized (effective) equation.

\begin{theorem}[Two-scale equation]
\label{thm:twosc}
Let $u_\de$ be the weak solution to \eqref{eq:weakHelmholtz}. There are functions $u\in H^1(G)$, $u_1\in L^2(\Om; H^1_{\sharp,0}(Y^*))$, and $u_2\in L^2(\Om; H^1_0(D)_\sharp)$ such that we have the  following two-scale convergences for $\de \to 0$
\begin{align*}
u_\de&\twosc u(x)+\chi_D(y)u_2(x,y),
&&&\chi_{\Om\setminus \overline{D}_\de}\nabla u_\de&\twosc\chi_{Y^*}(y)(\nabla u(x)+\nabla_y u_1(x,y)),\\
\de\chi_{D_\de}\nabla u_\de&\twosc\chi_D(y)\nabla_y u_2(x,y),
&&&\nabla u_\de&\twosc\nabla u \quad \text{ in }\quad G\setminus \overline{\Om}.
\end{align*} 
Here, the two-scale triple $\Vu:=(u, u_1, u_2)$ is the unique solution of 
\begin{equation}
\label{eq:twoscaleeq}
\begin{split}
\CB((u, u_1, u_2), (\psi, \psi_1, \psi_2))&=\int_{\partial G}g\psi^*\, d\si\\ &\hspace{-43pt}\forall \Vpsi:=(\psi, \psi_1, \psi_2)\in H^1(G)\times L^2(\Om; H^1_{\sharp, 0}(Y^*))\times L^2(\Om; H^1_0(D)_\sharp),
\end{split}
\end{equation}
with the two-scale sesquilinear form $\CB$ defined by
\begin{equation*}
\begin{split}
&\!\!\!\!\CB(\Vv, \Vpsi)\\
&:=\int_\Om \int_{Y^*}\ep_e^{-1}(\nabla v+\nabla_y v_1)\cdot (\nabla\psi^*+\nabla_y \psi_1^*)\, dy dx
+\int_\Om\int_D \ep_i^{-1}\nabla_y v_2\cdot \nabla_y \psi_2^*\, dydx\\
&\quad-k^2\int_G\int_Y (v+\chi_D v_2)(\psi^*+\chi_D \psi_2^*)\, dydx +\int_{G\setminus \overline{\Om}}\nabla v\cdot\nabla \psi^*\, dx -ik\int_{\partial G} v\psi^*\, d\si.
\end{split}
\end{equation*}
\end{theorem} 
The proof mainly follows the lines of \cite{BF04homhelmholtz} with the application of the two-scale convergences proved in \cite[Section 4]{All92twosc} for a highly heterogeneous diffusion problem.
Note that $u_1$ and $u_2$ are zero outside $\Om$ so that we have $u_\de\wto u$ in $H^1(G\setminus \overline{\Om})$.
We remark that the two-scale equation for a problem with highly heterogeneous coefficients includes two correctors and especially a corrector in the identity part -- in  contrast to the classical elliptic case, see \cite{All92twosc, LNW02twosc}.

The  two-scale equation can be re-cast into a homogenized macroscopic equation which involves effective parameters computed from cell problems, as given in the next theorem.

\begin{theorem}[Homogenized macroscopic equation]
\label{thm:twosccellandmacro}
$(u, u_1, u_2)$ solves the two-scale equation \eqref{eq:twoscaleeq} if and only if we set $u_1(x,y)=\sum_{j=1}^2\!\frac{\partial u}{\partial x_i}|_\Om(x)w_j(y)$, $u_2(x,y)=k^2 u|_\Om(x) w(y)$, and $u\in H^1(G)$ solves
\begin{equation}
\label{eq:homHelmholtz}
B_{\eff}(u, \psi)=\int_{\partial G}g\psi^*\, d\si\qquad \forall \psi\in H^1(G)
\end{equation}
with the effective sesquilinear form 
\begin{equation}
\label{eq:effsesquiform}
B_{\eff}(v, \psi):=\int_G a_{\eff}\nabla v\cdot \nabla \psi^*-k^2\mu_{\eff}v\psi^*\, dx-ik\int_{\partial G}v\psi^*\, d\si.
\end{equation}
Here, the effective parameters are defined as
\begin{align*}
(a_{\eff}(x))_{jk}&:=
\begin{cases}
\int_{Y^*}\ep_e^{-1}(\Ve_j+\nabla_y w_j)\cdot (\Ve_k+\nabla_y w_k^*)\, dy &\text{ if }x\in \Om\\
\Id_{jk}&\text{ if }x\in G\setminus\overline{\Omega}
\end{cases}\\
\text{and }\quad\mu_{\eff}(x)&:=
\begin{cases}
\int_Y1+k^2 w\chi_D\, dy&\text{ if }x\in \Om\\
1&\text{ if }x\in G\setminus \overline{\Omega},
\end{cases}
\end{align*}
where $w_j$ and $w$ are solutions to the following cell problems.
$w_j\in H^1_{\sharp, 0}(Y^*)$, $j=1,2$, solves
\begin{equation}
\label{eq:cella}
\int_{Y^*}\ep_e^{-1}(\Ve_j+\nabla_y w_j)\cdot \nabla_y \psi^*_1\, dy=0 \quad \forall \psi_1\in H^1_{\sharp,0}(Y^*)
\end{equation}
and $w\in H^1_0(D)_\sharp$ solves
\begin{equation}
\label{eq:cellmu}
\int_D\ep_i^{-1}\nabla_yw\cdot \nabla_y \psi_2^*-k^2w\psi_2^*\, dy=\int_D\psi_2^*\, dy \quad \forall \psi_2\in H^1_0(D)_\sharp.
\end{equation}
\end{theorem}

The presentation is orientated at the results for diffusion problems in \cite{All92twosc}, which can be seen most prominently in the form of the effective permeability $\mu_{\eff}$.
We prove that it is perfectly equivalent to the representation chosen in \cite{BF04homhelmholtz}, see Proposition \ref{prop:effective}.

The foregoing theorem means that in the limit $\de\to 0$, the scatterer $\Om$ can be described as a homogeneous material with the (effective) parameters $a_{\eff}$ (inverse permittivity) and $\mu_{\eff}$. 
Whereas $a_{\eff}$ is a positive definite matrix (see Proposition \ref{prop:effective}), the effective permeability $\mu_{\eff}$ exhibits some astonishing properties: First of all, its occurrence itself is surprising as the scatterer is non-magnetic. This is the already discussed effect of artificial magnetism.
Secondly, the permeability is frequency-dependent and its real part can have positive and negative sign. In the frequency region with $\Re(\mu_{\eff})<0$ waves cannot propagate leading to photonic band gaps, see \cite{BF04homhelmholtz}.
This effect is also studied numerically in detail in Section \ref{sec:numexperiment}.

We end with two observations on the two-scale equation, which are useful for the analysis later on.
We introduce the ``two-scale energy norm'' on $\CH:=H^1(G)\times L^2(\Om; H^1_{\sharp,0}(Y^*))\times L^2(\Om; H^1_0(D)_\sharp)$ as
\begin{align}
\label{eq:errornorm}
\|(v, v_1, v_2)\|^2_e:=\|\nabla v +\nabla_y v_1\|^2_{G\times Y^*}+\|\nabla_y v_2\|^2_{\Om\times D}+k^2\|v+\chi_D v_2\|^2_{G\times Y}.
\end{align}
In contrast to other homogenization settings, $\nabla v$ and $\nabla_y v_1$ as well as $v$ and $\chi_D v_2$ are no longer orthogonal. Still, the two-scale energy norm is equivalent to the natural norm of $\CH$, which is the statement of the next lemma.

\begin{lemma}
\label{lem:equivenergynorm}
The two-scale energy norm is equivalent to the natural norm of $\CH$
\[\|(v, v_1, v_2)\|_{\CH}^2:=\|v\|_{H^1(G)}^2+\|v_1\|_{L^2(\Om; H^1(Y^*))}^2+\|v_2\|_{L^2(\Om; H^1(D))}^2.\]
Furthermore, the two-scale energy norm is equivalent to the $k$-dependent norm
\[\|(v, v_1, v_2)\|_{k,\CH}^2:=\|v\|_{1,k,G}^2+\|v_1\|_{L^2(\Om; H^1(Y^*))}^2+\|v_2\|_{L^2(\Om; 1,k,D)}^2,\]
where the equivalence constants do not depend on $k$ and we have abbreviated  \[\|v_2\|^2_{L^2(\Om; 1,k,D)}:=\|\nabla_y v_2\|^2_{L^2(\Om; L^2(D))}+k^2\|v_2\|^2_{L^2(\Om; L^2(D))}.\]
\end{lemma}
\begin{proof}
The essential ingredient is a sharpened Cauchy-Schwarz inequality for the non-orthogonal terms
\begin{align*}
\Bigl|\int_G\int_{Y^*}\!\!\nabla v\cdot \nabla_y v_1\, dydx\Bigr|&\leq \|\nabla v\|_{L^2(G\times Y^*)}\|\nabla_y v_1\|_{L^2(G\times Y^*)}\\*
&=|Y^*|^{1/2}\|\nabla v\|_{L^2(G)}\|\nabla_y v_1\|_{L^2(\Om\times Y^*)}\\
\text{and }\Bigl|\int_G\int_{Y}v\chi_D v_2\, dydx\Bigr|&\leq \|v\|_{L^2(G\times D)}\|v_2\|_{L^2(G\times D)}=|D|^{1/2}\| v\|_{L^2(G)}\|v_2\|_{L^2(\Om; L^2(D))},
\end{align*}
where $|Y^*|, |D|<1$.
\end{proof}

\begin{lemma}
\label{lem:gaarding}
There exist constants $C_B>0$ and $C_{\min}:=\min\{1, \ep_e^{-1}, \Re(\ep_i^{-1})\}>0$ depending only on the parameters and the geometry, such that $\CB$ is continuous with constant $C_B$ and fulfills a G{\aa}rding inequality with constant $C_{\min}$, i.e.\
\begin{align*}
|\CB(\Vv, \Vpsi)|\leq C_B \|\Vv\|_e\|\Vpsi\|_e\quad \text{and}\quad \Re \CB(\Vv, \Vv)+2k^2\|v+\chi_D v_2\|_{G\times Y}^2\geq C_{\min}\|\Vv\|_e^2
\end{align*}
for all $\Vv:=(v, v_1, v_2), \Vpsi:=(\psi, \psi_1, \psi_2)\in \CH$.
\end{lemma}
\begin{proof}
The G{\aa}rding inequality is obvious from the definition of $\CB$ in Theorem \ref{thm:twosc}. The continuity of $\CB$ follows from the multiplicative trace inequality as in \cite{Mel95gFEM}.
\end{proof}

\subsection{Stability and regularity}
\label{subsec:analysis}
In this section, we derive stability and regularity results for the two-scale equation and its homogenized formulation. To achieve that goal, we analyze the cell problems and the macroscopic equation separately. Although the homogenized macroscopic equation is of Helmholtz-type, the unusual effective parameters introduce new aspects and challenges in the stability analysis.

\begin{proposition}
\label{prop:effective}
The effective parameters in $\Omega$ have the following properties:
\begin{enumerate}
\item $a_{\eff}$ is a real-valued, symmetric, uniformly elliptic matrix.
\item $\mu_{\eff}$ is a complex scalar with the upper bound on the absolute value
\begin{equation}
\label{eq:mueffupper}
|\mu_{\eff}|\leq C_\mu \qquad \text{with } C_\mu=C(\ep_i, D, Y, k_0).
\end{equation}
\item $\mu_{\eff}$ can be equivalently written as
\[\mu_{\eff}=1+\sum_{n\in \nz}\frac{k^2\varepsilon_i}{\lambda_n-k^2\varepsilon}\Bigl(\int_D\phi_n\, dx\Bigr)^2,\]
where $(\lambda_n, \phi_n)$ are the eigenvalues and eigenfunctions of the Laplace operator on $D$ with Dirichlet boundary conditions.
\item It holds that
\begin{equation}
\label{eq:muefflower}
\Im(\mu_{\eff})\geq C(\ep_i, D, Y)/k^2>0.
\end{equation} 
\end{enumerate}
\end{proposition}

The proof is postponed to Subsection \ref{subsec:homproof}.
The upper and lower bound on $\mu_{\eff}$ can only be obtained for $\Im(\ep_i)>0$. If we have an ideal lossless material (i.e.\ $\Im(\ep_i)=0$), $\mu_{\eff}$ is unbounded, see \cite{BF04homhelmholtz}.
As discussed above, the foregoing proposition shows that our $\mu_{\eff}$ agrees with the one presented in \cite{BF04homhelmholtz}.
However, we stress two advantages of our choice: 
First, it still holds for complex, but non-constant parameters $\varepsilon_i$.
Second, it only involves the solution of one cell problem rather than determining all eigenvalues and eigenfunctions of the Dirichlet Laplacian, which is very useful for the numerical implementation.
The lower bound on $\Im(\mu_{\eff})$ might be improved using sophisticated methods for estimating eigenvalues and averages of eigenfunctions of the Dirichlet Laplacian. 
We emphasize that our numerical experiment from Section \ref{sec:numexperiment} does not show this severe $k$-dependence of the lower bound.

For the properties of the effective parameters, the cell problems have already been implicitly analyzed. Hence, results on the two-scale corrections $u_1$ and $u_2$ follow immediately.

\begin{proposition}
\label{prop:stabilitycell}
There are $C_{\stab, 1}, C_{\stab,2}>0$ depending only on $\ep_i^{-1}$, $\ep_e^{-1}$, $D$, $Y^*$, and $k_0$, such that the correctors $u_1$ and $u_2$ satisfy
\[\|u_1\|_{L^2(\Om; H^1(Y^*))} \leq C_{\stab, 1}\|\nabla u\|_{G}\quad
\text{ and }\quad
 \|u_2\|_{L^2(\Om; 1,k, D)}\leq C_{\stab, 2} \|u\|_{1,k, G}\]
with the notation $\|\cdot\|_{L^2(\Om; 1,k,D)}$ explained in Lemma \ref{lem:equivenergynorm}.
\end{proposition}

All elements of the two-scale solution triple admit higher regularity depending on the geometry.
\begin{proposition}
\label{prop:regularity}
Let $g\in H^{1/2}(\partial G)$. There are regularity coefficients $s(\Om, G)$, $s(Y^*)$, and $s(D)$ with $s(\cdot)\in (\halb, 1]$ such that
\begin{enumerate}
\item  for all $0<s<s(D)$, $u_2\in L^2(\Om; H^{1+s}(D))$ with $\|u_2\|_{L^2(\Om; H^{1+s}(D))}\leq C_{\reg, 2}\,k\|u\|_{1,k,\Om}$;
\item for all $0<s<s(Y^*)$, $u_1\in L^2(\Om; H^{1+s}(Y^*))$ with $\|u_1\|_{L^2(\Om; H^{1+s}(Y^*))}\!\leq C_{\reg, 1}\,\|\nabla u\|_\Om$;
\item for all $0<s<s(\Om, G)$, $u\in H^{1+s}_{pw}(G)$ with 
\begin{align}
\label{eq:regularityu}
\|u\|_{H^{1+s}_{pw}(G)}\leq C(k\|u\|_{1,k,G}+\|f\|_G+\|g\|_{H^{1/2}(\partial G)}).
\end{align}
\end{enumerate}
\end{proposition}

\begin{proof}
The first two points follow from classical elliptic regularity theory. For the estimate of $u_2$ the term $k^2u_2$ is treated as additional right-hand side thereby leading to the additional factor $k$. Confer similar higher regularity estimates for the classical Helmholtz equation as in \cite{Mel95gFEM}, for instance.  
The result for $u$ can also be deduced using regularity for elliptic interface problems, see \cite{Petz10reginterface}. In this case, the term $k^2\mu_{\eff}u$ is interpreted as an additional right-hand side. 
\end{proof}

With $C^2$ boundary of $D$ (and then also $Y^*$), we obtain $s(D)=s(Y^*)=1$. For the numerical treatment, $D$ is approximated by a polygonally bounded Lipschitz domain.
As also $\partial \Om$ is of class $C^2$ and $G$ is convex, we have $s(\Om, G)=1$. The interface $\partial \Om$ is also approximated by a piecewise polygonal interface in practical numerical schemes. 
In general, the maximal regularity  of the problems posed on a polygonal Lipschitz domain depends on the domain's maximal interior angle, see \cite{Petz10reginterface}. 
We give the regularity results in their general form as polygonal (non-convex) domains have to be considered in the process of boundary approximation in Sections \ref{sec:HMM} and \ref{sec:quasiopt}.

Looking at estimate \eqref{eq:regularityu}, we note that we need an estimate  for $\|u\|_{1,k,G}$ in terms of the data. 
From Fredholm theory we have a stability estimate of the form $\|u\|_{1,k,G}\leq C(k) \|g\|_{\partial G}$, but the dependence of the constant on the wavenumber $k$ is unknown. We therefore make the following assumption of polynomial stability.

\begin{assumption}
\label{asspt:polynomialstable}
Assume that there is $q\in \nz_0$ and $C_{\stab,0}>0$ such that the solution $u$ to \eqref{eq:homHelmholtz} with additional right-hand side $f\in L^2(G)$  fulfills
\[\|u\|_{1,k,G}\leq C_{\stab,0}\, k^q(\|f\|_G+\|g\|_{H^{1/2}(\partial G)}).\]
\end{assumption}
Polynomial stability is not trivial: There are  so called trapping domains leading to exponential growth of the  stability estimate in $k$, see \cite{BCGLL11trapping}.
In our setting, we can prove the assumption with $q=3$ under some (mild) additional assumptions.  More explicitly speaking, we have the following theorem, which is proved in Subsection \ref{subsec:stabilityproof}.

\begin{theorem}[Stability]
\label{thm:stabilityeff}
Assume that there is $\gamma>0$ such that
\begin{align}
\label{eq:geometry}
x\cdot n_G\geq \gamma \text{ on }\partial G &&x\cdot n_\Om\geq 0\text{ on }\partial \Om,
\end{align}
where $n$ denotes the outer normal of the domain specified in the subscript. Furthermore assume that $a_{\eff}|_{G\setminus \overline{\Om}}-a_{\eff}|_\Om$ is negative semi-definite. 
Let $u$ be the solution to \eqref{eq:homHelmholtz} with additional volume term $\int_G f\phi^*\, dx$ on the right hand-side for $f\in L^2(G)$. Then there is $C_{\stab,0}$ only depending on the geometry, the parameters, and $k_0$, such that $u$ satisfies the stability estimate
\[ \|u\|_{1,k,G}\leq C_{\stab, 0}(k^3\|f\|_{G\setminus\overline{\Om}}+k^2\|f\|_\Om+k^{3/2}\|g\|_{\partial G}+k^{-1}\|g\|_{H^{1/2}(\partial G)}).\] 
\end{theorem}

The geometrical assumption \eqref{eq:geometry} is the common assumption for scattering problems, see \cite{EM12helmholtz, Het07helmholtz, MS14signindefinite}. It can, for example, be fulfilled if $\Om$ is convex (and w.l.o.g.\ $0\in \Om$) and $G$ is chosen appropriately. 
The assumption on $a_{\eff}$ in fact is an assumption on $\ep_e$ and can be fulfilled for appropriate choices of material inside and outside the scatterer.
Analytically, this assumption can be traced back to the assumption that ``$Da\cdot x$ is negative semi-definite'' for Lipschitz continuous $a$ in Proposition \ref{prop:stabilitylipschitz}. In order to obtain that proposition, a weaker condition on the Lipschitz constant of $a$ would be sufficient, but then the constant in the stability estimate would depend on the Lipschitz constant of $a$, which blows up in the approximation of $a_{\eff}$.
We emphasize that a similar condition on the derivative of the diffusion coefficient and/or its Lipschitz constant has also been imposed in the scalar case in \cite{BGP15hethelmholtzLOD}. 

In the literature, most stability results for Helmholtz problems have been obtained in the case of constant coefficients, see e.g.\ \cite{BSW16helmholtz, EM12helmholtz, Het07helmholtz, Mel95gFEM, MS11helmholtz, MS14signindefinite}.
Only recently scalar-valued, Lipschitz continuous real-valued heterogeneous coefficients have been studied in \cite{BGP15hethelmholtzLOD}.
All these works have obtained the stability estimate with $q=0$ under the same geometry assumption \eqref{eq:geometry} as here. 
Our setting exhibits three new challenges for the stability analysis: A discontinuous, namely piece-wise constant, diffusion coefficient, a partly complex parameter $\mu$ and the fact that the diffusion coefficient $a$ is matrix-valued.
The second aspect introduces the worse dependence on $k$ in the stability estimate, as explained after Proposition \ref{prop:stabilitylipschitz}.
There, we also discuss how the lower bound on the imaginary part of $\mu$ influences the stability estimate.

Under the assumption of polynomial stability, the (final) stability and regularity estimates for the two-scale equation are deduced. A bound on the inf-sup-constant of the corresponding sesquilinear form is obtained similar to \cite{Het07helmholtz, Mel95gFEM, P15LODhelmholtz}.

\begin{proposition}
\label{prop:infsupconst}
If Assumption \ref{asspt:polynomialstable} is satisfied, the following holds:
\begin{enumerate}
\item The two-scale solution  satisfies 
\[\|(u, u_1, u_2)\|_e\leq C_{\stab, e}\,k^q(\|f\|_G+\|g\|_{H^{1/2}(\partial G)})\]
for $C_{\stab,e}:=C_{\stab, 0}(1+C_{\stab, 1}+C_{\stab, 2})$.
\item The regularity estimate for $u$ is
\[\|u\|_{H^{1+s}_{pw}(G)}\leq C_{\reg, 0}\,k^{q+1}(\|f\|_G+\|g\|_{H^{1/2}(\partial G)}).\]
\item The inf-sup-constants of $B_{\eff}$ and $\CB$ can be bounded below as follows
\begin{align}
\label{eq:infsupBeff}
\inf_{v\in H^1(G)}\sup_{\psi\in H^1(G)}\frac{\Re B_{\eff}(v, \psi)}{\|v\|_{H^1(G)}\|\psi\|_{H^1(G)}}&\geq C_{\inf, \eff} k^{-(q+1)},\\
\label{eq:infsuptwosc}
\inf_{\Vv\in \CH}\sup_{\Vpsi\in\CH}\frac{\Re \CB(\Vv, \Vpsi)}{\|\Vv\|_e\|\Vpsi\|_e}&\geq C_{\inf, e} k^{-(q+1)}
\end{align}
with $C_{\inf, \eff}:=\min\{ \alpha, C_\mu\} (k_0^{-(q+1)}+C_{\stab, 0})^{-1}$, where $\alpha$ denotes the ellipticity constant of $a_{\eff}$, and $C_{\inf, e}:=\min\{C_{\min}, 1\}(k_0^{-(q+1)}+C_{\stab, e})^{-1}$.
\end{enumerate}
\end{proposition}

\section{The Heterogeneous Multiscale Method}
\label{sec:HMM}
As explained in the introduction, a direct discretization of the heterogeneous problem \eqref{eq:weakHelmholtz} is infeasible due to the necessary small grid mesh width resolving all inclusions.
The idea of the Heterogeneous Multiscale Method is to imitate the homogenization process and to thereby provide a method based on grids independent of the finescale parameter $\delta$. 
In this paper, we introduce the HMM as a direct discretization of the two-scale equation \eqref{eq:twoscaleeq}, see \cite{Ohl05HMM} for the original idea for elliptic diffusion problems. 
This point of view is vital for the numerical analysis in Section \ref{sec:quasiopt} since ideas and procedures developed for ``normal'' Helmholtz problems can be easily transferred.
However, we will also shortly explain below how this direct discretization can be decoupled into macroscopic and microscopic computations in the fashion of the HMM as originally presented in \cite{EE03hmm, EE05hmm}. 

In this and the next section, we assume that $D$ and $\Om$ are polygonally bounded (in contrast to the $C^2$ boundaries in the analytic sections).
The reason is that the $C^2$ boundaries can be approximated by a series of more and more fitting polygonal boundaries. 
This procedure of boundary approximation results in non-conforming methods, i.e.\ the discrete function spaces are no subspaces of the analytic ones.
We avoid this difficulty in our numerical analysis by assuming polygonally bounded domains by now. The new assumption reduces the possible higher regularity of solutions as discussed in Subsection \ref{subsec:analysis}. However, we can always obtain the maximal regularity in the limit of polygonal approximation of $C^2$ boundaries, which we have in mind as application case.

Denote by $\CT_H=\{ T_j|j\in J\}$ and $\CT_h=\{S_k|k\in I\}$ conforming and shape regular triangulations of $G$ and $Y$, respectively. Additionally, we assume that $\CT_H$ resolves the partition into $\Om$ and $G\setminus \overline{\Om}$ and that $\CT_h$ resolves the partition of $Y$ into $D$ and $Y^*$ and is periodic in the sense that it can be wrapped to a regular triangulation of the torus (without hanging nodes). 
We define the local mesh sizes $H_j:=\diam(T_j)$ and $h_k:=\diam(S_k)$ and the global mesh sizes $H:=\max_{j\in J}H_j$ and $h:=\max_{k\in I}h_k$.
Finally, the discrete function spaces $V_H^1\subset H^1(G)$, $\widetilde{V}_h^1((Y^*)_j^\de)\subset  H^1_{\sharp, 0}((Y^*)_j^\de)$, and $V_h^1(D_j^\de)\subset H^1_0(D_j^\de)_\sharp$ are defined as 
\begin{align*}
V_H^1&:=\{v_H\in H^1(G)|v_H|_T\in \pz^1\quad \forall T\in \CT_H\}\\
\widetilde{V}_h^1((Y^*)_j^\de)&:=\{v_h\in H^1_{\sharp, 0}((Y^*)_j^\de)|v_h|_S\in \pz^1 \quad \forall S\in \CT_h((Y^*)_j^\de)\}\\
V_h^1(D_j^\de)&:=\{v_h\in H^1_0(D_j^\de)_\sharp|v_h|_S\in \pz^1 \quad \forall S\in \CT_h(D_j^\de)\},
\end{align*}
where $\pz^1$ are the polynomials of maximal degree $1$.

\begin{definition}
\label{def:hmmdiscrtwosc}
The discrete two-scale solution 
\[(u_H, u_{h, 1}, u_{h,2})\in V_H^1\times L^2(\Om; \widetilde{V}_h^1(Y^*))\times L^2(\Om;V_h^1(D))\]
is defined as the solution of
\begin{align}
\label{eq:discrtwosc}
\CB((u_H, u_{h,1}, u_{h,2}), (\psi_H, \psi_{h,1}, \psi_{h,2}))&=\int_{\partial G}g\,\psi_H^*\, d\si\\
\nonumber
&\hspace{-103pt} \forall (\psi_H, \psi_{h,1}, \psi_{h,2})\in V_H^1\times L^2(\Om; \widetilde{V}_h^1(Y^*))\times L^2(\Om; V_h^1(D))
\end{align}
with the two-scale sesquilinear form $\CB$ defined in Theorem \ref{thm:twosc}.
\end{definition}
In order to evaluate the integrals over $G$ in $\CB$, one introduces quadrature rules, which are exact for the given ansatz and test spaces.
In our case of piecewise linear functions, it suffices to choose the one-point rule $\{|T_j|, x_j\}$ with the barycenter $x_j$ for the gradient part and a second order quadrature rule $Q_j^{(2)}:=\{q_l, x_l\}_l$ with $l=1,2,3$ for the identity part.
As a consequence, the functions $u_{h,1}$ and $u_{h,2}$ will also be discretized in their part depending on the macroscopic variable $x$:
In fact, one has $u_{h,1}\in S_H^0(\Om;\widetilde{V}_h^1(Y^*))$ and $u_{h,2}\in S_H^1(\Om; V_h^1(D))$.
Here, the space of discontinuous, piecewiese $p$-polynomial (w.r.t.\ $x$) discrete functions is defined as
\begin{align*}
S_H^p(\Om; X_h)&:=\{v_h\in L^2(\Om; X)|\,v_h(\cdot, y)|_{T_j}\in\pz^p\; \forall j\in J, y\in Y; v_h(x, \cdot)\in X_h\;\forall x\in \Om\},
\end{align*}
for any conforming finite element space $X_h\subset X$.
Note that $u_{h,2}$ is a piecewise $x$-linear discrete function, since $Q^{(2)}$ consists of $3$ quadrature points on each triangle.

The functions $u_{h,1}$ and $u_{h,2}$ are the discrete counterparts of the analytical correctors $u_1$ and $u_2$.
They are correctors to the macroscopic discrete function $u_H$ and solve discretized cell problems.
These cell problems, posed in the unit square $Y$, can be transferred back to $\delta$-scaled and shifted unit squares $Y_j^\delta=x_j+\delta Y$, where $x_j$ is a  macroscopic quadrature point. 
This finally gives an equivalent formulation of \eqref{eq:discrtwosc} in the form of a (traditional) HMM.
The formulation using a macroscopic sesquilinear form with local cell reconstructions is used in practical implementations.
We emphasize that the presented HMM also works for locally periodic $\varepsilon^{-1}$ depending on $x$ and $y$.
The HMM and its interpretation as discretization of a fully coupled two-scale equation can even be applied to non-periodic problems, as demonstrated in \cite{HO15hmmmonotone}.

\section{Quasi-optimality of the HMM}
\label{sec:quasiopt}
Based on the definition of the HMM in Definition \ref{def:hmmdiscrtwosc}, we analyze its quasi-optimality in Theorem \ref{thm:infsupdiscr}. This quasi-optimality is a kind of C{\'e}a lemma for indefinite sesquilinear forms and directly leads to a priori estimates.

All estimates will be derived in the ``two-scale energy norm'' \eqref{eq:errornorm}.
Let us furthermore define the error terms $e_0:=u-u_H$, $e_1:=u_1-u_{h,1}$, and $e_2:=u_2-u_{h,2}$. We will only estimate these errors and leave the modeling error introduced by homogenization apart.
Recall the abbreviation $\CH:=H^1(G)\times L^2(\Om; H^1_{\sharp,0}(Y^*))\times L^2(\Om; H^1_0(D)_\sharp)$. In a similar short form we write $\VV_{H,h}:=V_H^1\times L^2(\Om; \widetilde{V}_h^1(Y^*))\times L^2(\Om; V_h^1(D))$.

We recall that the finite element function space $\VV_{H,h}$ has the following approximation property: There is $C_{\appr}$ such that for all $\halb<s\leq 1$  and given $(v,v_1,v_2)\in H^{1+s}_{pw}(G)\times L^2(\Om; H^{1+s}(Y^*))\times L^2(\Om; H^{1+s}(D))$  it holds
\begin{equation}
\label{eq:approxvHh}
\begin{aligned}
(\|v-v_H\|_G+H\|\nabla(v-v_H)\|_G)&\leq C_{\appr} H^{1+s}|v|_{H^{1+s}_{pw}(G)},\\
(\|v_1-v_{h,1}\|_{\Om\times Y^*}+h\|\nabla_y(v_1-v_{h,1})\|_{\Om\times Y^*})&\leq C_{\appr} h^{1+s}|v_1|_{L^2(\Om; H^{1+s}(Y^*))},\\
(\|v_2-v_{h,2}\|_{\Om\times D}+h\|\nabla_y(v_2-v_{h,2})\|_{\Om\times D})&\leq C_{\appr} h^{1+s}|v_2|_{L^2(\Om;H^{1+s}(D))}
\end{aligned}
\end{equation}
for all $\Vv_{H,h}:=(v_H, v_{h,1}, v_{h,2})\in \VV_{H,h}$.
Note that the regularity coefficient $s$ does not necessarily have to be the same in all three estimates.

In the $h$-version of the Finite Element method we consider in this paper, the meshes $\CT_H$ and $\CT_h$ are refined (thus decreasing $H$ and $h$) in order to obtain a better approximation. 
Hence, we introduce constants $H_{\max}>0$ and $h_{\max}>0$ such that $H\leq H_{\max}$ and $h\leq h_{\max}$ for all considered grids.

\begin{theorem}[Discrete inf-sup-stability and quasi-optimality]
\label{thm:infsupdiscr}
Let Assumption \ref{asspt:polynomialstable} be satisfied and let $s(\Om, G)$, $s(Y^*)$, and $s(D)$ be the (higher) regularity exponents from Proposition \ref{prop:regularity}. Fix $(s_0, s_1, s_2)$ with $0<s_0<s(\Om, G)$, $0<s_1<s(Y^*)$, $0<s_2<s(D)$.
If the wave number $k$ and the mesh widths $H$, $h$ are coupled by
\begin{equation}
\label{eq:resolutionasspt}
\begin{aligned}
k^{q+2}H^{s_0}&\leq -\frac{k_0^{q+1}}{2H_{\max}^{1-s_0}}+\sqrt{\frac{k_0^{q+1}}{H_{\max}^{1-s_0}}\Bigl(\frac{C_{\min}}{12C_BC_{\appr}C_{\reg,0}}+\frac{k_0^{q+1}}{4H_{\max}^{1-s_0}}\Bigr)},\\
k^{q+1}h^{s_1}&\leq \frac{C_{\min}}{12C_BC_{\appr}C_{\reg, 1}C_{\stab, e}},\\
k^{q+2}h^{s_2}&\leq -\frac{k_0^{q+1}}{2h_{\max}^{1-s_2}}+\sqrt{\frac{k_0^{q+1}}{h_{\max}^{1-s_2}}\Bigl(\frac{C_{\min}}{12C_BC_{\appr}C_{\reg,2}C_{\stab, e}}+\frac{k_0^{q+1}}{4h_{\max}^{1-s_2}}\Bigr)},
\end{aligned}
\end{equation}
then
\begin{equation}
\label{eq:discrinfsup}
\inf_{\Vv_{H, h}\in \VV_{H,h}}\sup_{\Vpsi_{H, h} \in \VV_{H,h}}\frac{\Re\CB(\Vv_{H, h}, \Vpsi_{H, h})}{\|\Vv_{H, h}\|_e\, \|\Vpsi_{H, h}\|_e}\geq \frac{C_{\HMM}}{k^{q+1}}
\end{equation}
with $C_{\HMM}:=\frac{C_{\min}}{2}(k_0^{-(q+1)}(1+\frac{C_{\min}}{2C_B})+C_{\stab, e})^{-1}$ and the error between the two-scale solution and the HMM-approximation satisfies
\begin{equation}
\label{eq:hmmapriori}
\|(e_0, e_1, e_2)\|_e  \leq \frac{2C_B}{C_{\min}} \inf_{\Vv_H\in \VV_{H, h}}\!\|\Vu-\Vv_H\|_e\leq C((H^{s_0}+h^{s_2})k^{q+1}+k^qh^{s_1})\|g\|_{H^{1/2}(\partial G)}.
\end{equation}
\end{theorem}
The proof is postponed to Subsection \ref{subsec:proofquasiopt}.

\begin{corollary}
Under the maximal possible regularity $s_0=s_1=s_2=1$ as discussed in Subsection \ref{subsec:analysis}, the energy error  converges with rate $k^{q+1}(H+h)$ under the resolution assumption that $k^{q+2}(H+h)$ is sufficiently small.
\end{corollary}

Dual problems can be used to estimate\hfill $\|(e_0, e_1, e_2)\|_{L^2}$\hfill by \hfill $C(k^{q+1}(H^{s_0}+h^{s_2})+k^qh^{s_1})\\ \|(e_0,e_1,e_2)\|_e$ as in the the proof of Theorem \ref{thm:infsupdiscr}. 
This is the classical Aubin-Nitsche argument to obtain higher convergence rates in the $L^2$-norm, for details see \cite{MS14helmholtzl2, MS11helmholtz} for classical Helmholtz problems. 

As it has already been remarked in \cite{HOV15maxwellHMM, Ohl05HMM}, the definition of the HMM as direct discretization of the two-scale equation, see \eqref{eq:discrtwosc}, is the crucial starting point for all kinds of error estimates and in particular, enables us to derive a posteriori error estimates.
 This can also be achieved for the setting considered here by adapting a posteriori error estimates for Helmholtz problems obtained e.g.\ in \cite{DS13helmholtzaposteriori, IB01helmholtzaposteriori} to the two-scale equation.

\smallskip

Under the regularity estimate from Assumption \ref{asspt:polynomialstable}, the resolution condition \eqref{eq:resolutionasspt} is optimal / unavoidable for standard finite element methods and the multiscale setting: 
As the second cell problem depends on $k$, it is natural that $h$ enters the condition \eqref{eq:resolutionasspt}. We emphasize that $h$ denotes the mesh width of the unit square mesh and is thus not coupled to $\delta$ in any way.
 Assuming now $q=0$, as it is the case for classical Helmholtz problems, we regain the usual condition ``$k^2(H+h)$ sufficiently small'', cf.\ e.g.\ \cite{EM12helmholtz, Het07helmholtz, Ihl98, Mel95gFEM, MS11helmholtz}, see also the early abstract discussion in \cite{Sch74galerkinindefinite}.
 This is also the resolution condition we experience in our numerical experiments in Section \ref{sec:numexperiment}.
Our explicit stability estimate in Theorem \ref{thm:stabilityeff} yields $q=3$ and thus, the resolution condition ``$k^5(H+h)$ small''.
This is a kind of ``worst case'' resolution condition: It is certainly sufficient for the a quasi-optimality and a priori error result presented above, but can well (as the numerical example indicates) be sub-optimal.
We emphasize that this gap between the optimal and worst-case resolution condition is no defect of the numerical method, but can be closed if better stability results in the spirit of Theorem \ref{thm:stabilityeff} are proved, which is outside the scope of our work.

As also supported by our numerical experiment, the HMM is much more efficient than a direct discretization of the heterogeneous Helmholtz problem \eqref{eq:weakHelmholtz}. 
In order to get an accurate solution, one needs a grid with mesh size $h_{\mbox{\tiny{ref}}}$ satisfying $h_{\mbox{\tiny{ref}}}<\de\ll 1$ from the multiscale point of view. On top of that, at least $k^2h_{\mbox{\tiny{ref}}}<C$ has to be satisfied to rule out pre-asymptotic effects. 
Note that the heterogeneous problem does not fulfill the assumptions for any available stability estimate, so that the resolution condition may even be worse.

Although the so-called pollution effect is not avoidable for the classical Helmholtz equation in dimension $d\geq2$ as shown in \cite{BS00pollutionhelmholtz}, much work in its reduction has been invested: 
Examples of the proposed methods are the $hp$-version of the finite element method \cite{EM12helmholtz, MS11helmholtz}, (hybridizable) discontinuous Galerkin methods \cite{CLX13HDGhelmholtz, GM11HDGhelmholtz}, or plane wave Trefftz methods \cite{HMP16surveytrefftz,HMP16pwdg, PPR15pwvem}, just to name a few.
Recently, it has been shown that the resolution condition can be relaxed to the natural assumption ``$kh$ sufficiently small'' by applying a Localized Orthogonal Decomposition (LOD) to the Hemholtz equation, see \cite{BGP15hethelmholtzLOD, GP15scatteringPG, P15LODhelmholtz}. 
The function space is decomposed into a coarse space, where the solution is sought, and a remainder space. The coarse space is spanned by pre-computable basis functions with local support, which include some information from the remainder space by the solution of localized correction problems.
The definition of the HMM as direct discretization of the two-scale equation makes it possible to apply an additional LOD, see \cite{OV16hmmlod2}.

\section{Main proofs}
\label{sec:proofs}
In this section the essential proofs of the properties of the effective parameters occurring in homogenization, the stability of the effective equation and the quasi-optimality of the HMM will be given.

\subsection{Proof of the properties of the effective parameters}
\label{subsec:homproof}
In this section we show the upper and lower bounds for the effective permeability $\mu_{\eff}$.
We also show the equivalence of the two formulations of $\mu_{\eff}$ obtained from Allaire \cite{All92twosc} and Bouchitt{\'e} and Felbacq \cite{BBF09hom3d}, respectively.

\begin{proof}[Proof of Proposition \ref{prop:effective}]
The characterization of $a_{\eff}$ is well-known and follows from the ellipticity of the corresponding cell problem \eqref{eq:cella}, see \cite{All92twosc} for similar cell problems.

Cell problem \eqref{eq:cellmu} is (uniformly) coercive because of $\Im(\varepsilon^{-1}_i)<0$.
The Lax-Milgram-Babu{\v{s}}ka theorem \cite{Bab70fem} now implies the unique solvability of the cell problem for $w$ with the stability estimate
\[\|w\|_{1,k,D}\leq C(\varepsilon_i, k_0, D)/k.\]
Combination with the representation of $\mu_{\eff}$ directly yields \eqref{eq:mueffupper}.

It is well-known that the eigenfunctions of the Laplace operator on $D$ with Dirichlet boundary conditions form an orthonormal basis of $L^2(D)$. The eigenvalues $\lambda_n$ are sorted as a positive, increasing sequence of real numbers. 
We have the representation $1=\sum_n\Bigl(\int_D \phi_n\Bigr)\phi_n $.
Writing $w=\sum_n\alpha_n\phi_n$ and inserting this into \eqref{eq:cellmu}, gives after a comparison of coefficients
\[w =\sum_n \Bigl(\frac{\ep_i}{\lambda_n-k^2\ep_i}\int_D\phi_n\Bigr)\phi_n 
\qquad \text{and hence,} \qquad
\mu_{\eff} = 1+\sum_n\frac{k^2\ep_i}{\lambda_n-k^2\ep_i}\Bigl(\int_D\phi_n\Bigr)^2,\]
see \cite{BF04homhelmholtz}. 
A similar computation for the full three-dimensional case is given in \cite[Appendix A]{LS15negindex}. 
Now we can deduce because of the positivity of $\Im(\ep_i)$ and of the eigenvalues that
\begin{align*}
\Im(\mu_{\eff})=\sum_n\frac{k^2\lambda_n\Im(\ep_i)}{|\lambda_n-k^2\ep_i|^2}\Bigl(\int_D\phi_n\Bigr)^2
\geq \frac{k^2\lambda_0\Im(\ep_i)}{|\lambda_0-k^2\ep_i|^2}\Bigl(\int_D\phi_0\Bigr)^2.
\end{align*}
The first eigenfunction of the Dirichlet Laplacian is zero-free, thus $(\int_D\phi_0)^2>0$. As we consider the high-frequency case, we can w.l.o.g.\ assume $\lambda_0\leq k^2|\ep_i|$ and then obtain $|\lambda_0-k^2\ep_i|^2\leq 2k^4|\ep_i|^2$. 
This finally gives
\[\Im(\mu_{\eff})\geq \frac{k^2\lambda_0\Im(\ep_i)}{2k^4|\ep_i|^2}\Bigl(\int_D\phi_0\Bigr)^2\geq \frac{C(\ep_i, D)}{k^2}>0.\]
\end{proof}

\subsection{Polynomial stability of the Helmholtz equation with discontinuous coefficients}
\label{subsec:stabilityproof}
In this section, we give a detailed proof of Theorem \ref{thm:stabilityeff}. 
We consider a Lipschitz continuous, matrix-valued diffusion coefficient $a$ with the partly complex-valued $\mu$ first. 
Then the discontinuity in $a_{\eff}$ is treated by a smoothing/approximation procedure. A direct application of the Rellich-Morawetz identities (see e.g.\ \cite[Section 2]{MS14signindefinite} and the references therein) is not possible due to jumps in the gradient of the solution over the interface.

\begin{proposition}
\label{prop:stabilitylipschitz}
Let $\Om$ and $G$ satisfy \eqref{eq:geometry}. Let $u$ be the unique solution to
\[B(u, \psi)=(f, \psi)_G+(g, \psi)_{\partial G}\]
for $f\in L^2(G)$ and $g\in L^2(\partial G)$, where $B$ is the sesquilinear form of \eqref{eq:effsesquiform} with $a_{\eff}$ replaced by $a$ and $\mu_{\eff}$ replaced by $\mu$ fulfilling the assumptions
\begin{itemize}
\item $a\in W^{1,\infty}(G, \mathbb{R}^{2\times 2})$ is symmetric, bounded and uniformly elliptic;
\item the matrix $Da\cdot x$ with $(Da\cdot x)_{ij}:=\sum_k x_k\,\partial_k a_{ij} $ is negative semi-definite;
\item $\mu\in L^\infty(G;\mathbb{C})$ is piecewise constant, namely  $\mu=\mu_2\in \rz_+$ in $G\setminus\overline{\Om}$ and $\mu=\mu_1\in \cz$ in $\Om$ with $\Im(\mu_1)>c_0>0$.
\end{itemize}
Then the following stability estimate holds
\begin{align*}
 \|u\|_{1,k,G}&\leq Ck^{1/2}(c_0^{-1/2}\!+1)\|g\|_{\partial G}+C\|f\|_G+C(c_0^{-1/2}\!+c_0^{-1})\|f\|_\Om\\*
 &\quad+\frac{C}{k}(1+c_0^{-1/2}+c_0^{-1})\|f\|_G+\frac{Ck}{c_0}\|f\|_{G\setminus\overline{\Om}},
\end{align*}
where the constants depend on the geometry, the upper bounds on $\mu$ and $a$, the ellipticity constant of $a$, and on $k_0$; but not on the Lipschitz constant of $a$ or any other constant involving the derivative of $a$.
\end{proposition}

\begin{proof}
{\itshape First step:} With $\psi=u$ and considering the imaginary part, we obtain with H\"older and Young's inequality
\begin{align}
\label{subeq:imucase3}
k^2c_0\|u\|^2_\Om+k\|u\|^2_{\partial G}&\leq C \Bigl(\frac{1}{k}\|g\|^2_{\partial G}+\frac{1}{k^2 c_0}\|f\|^2_\Om+\|f\|_{G\setminus \overline{\Om}}\|u\|_{G\setminus \overline{\Om}}\Bigr).
\end{align}

{\itshape Second step:} With $\psi=u$ and considering the real part, we obtain due to the boundedness of $\mu$ and the uniform ellipticity of $a$
$$\|\nabla u\|^2_G\leq C\Bigl(k^2\|u\|^2_G+\frac{1}{2k^2}\|f\|^2_G+\frac{k^2}{2}\|u\|_G^2+\|g\|_{\partial G}\|u\|_{\partial G}\Bigr).$$
Inserting \eqref{subeq:imucase3} yields
\begin{align}
\label{subeq:realucase3}
\|\nabla u\|^2_G&\leq C\Bigl(k^2\|u\|^2_{G\setminus\overline{\Om}}+\frac{1}{k^2}\Bigl(1+\frac{1}{c_0^2}\Bigr)\|f\|_G+\frac{1}{k^2 c_0}\|f\|^2_\Om
+\frac{1}{k}\Bigl(\frac{1}{c_0}+1\Bigr)\|g\|^2_{\partial G}\Bigr).
\end{align}

{\itshape Third step:} It remains to estimate $\|u\|^2_{G\setminus \overline{\Om}}$. For this, we insert $\psi=x\cdot \nabla u$ and consider the real part. Note that $x\cdot \nabla u$ is an admissible test function because we have $u\in H^2(G)$ due to the convexity of $G$ and the smoothness of $a$, see \cite{GilTru}. We moreover use the identity $\partial_j(|w|^2)=2\Re(w \partial_j w^*)$. For the first term of the sesquilinear form we obtain
\begin{align*}
&\!\!\!\!\Re\int_Ga\nabla u\cdot \nabla(x\cdot \nabla u^*)\, dx\\
&=\Re\int_Ga\nabla u\cdot \nabla u^*+a\nabla u\cdot (D^2u^*) x\, dx\\
&=\int_Ga\nabla u\cdot \nabla u^*+\halb \nabla (a\nabla u\cdot \nabla u^*)\cdot x -\halb (Da\cdot x)\nabla u\cdot \nabla u^*\, dx\\*
&=-\halb \int_G(Da\cdot x)\nabla u\cdot \nabla u^*\, dx+\halb \int_{\partial G}a \nabla u\cdot \nabla u^*x\cdot n\, d\sigma,
\end{align*}
where in last equality we integrated by parts. As $Da\cdot x$ is negative semi-definite by the assumption, the first term is non-negative.

For the second part of the sesquilinear form we obtain 
\begin{align*}
&\!\!\!\!\Re\int_G k^2 \mu ux\cdot \nabla u^*\, dx\\*
&=\Re\int_\Om k^2\mu_1 ux\cdot \nabla u^*\, dx+\frac{\mu_2}{2}\int_{G\setminus \overline{\Om}}k^2 x\cdot\nabla|u|^2\, dx\\
&=\Re\int_\Om k^2\mu_1 ux\!\cdot\! \nabla u^*\, dx+\frac{\mu_2}{2} \int_{\partial(G\setminus \overline{\Om})}\!\!\!\!k^2|u|^2 x\!\cdot \!n\, d\sigma-\int_{G\setminus \overline{\Om}}\!k^2\mu_2 |u|^2\, dx.
\end{align*}
So for the test function $\psi=x\cdot \nabla u$ and the real part we deduce by combining the foregoing calculations
\begin{align*}
&\!\!\!\!\halb \int_{\partial G}a\nabla u\cdot\nabla u^* x\cdot n\, d\sigma+\int_{G\setminus \overline{\Om}}k^2\mu_2 |u|^2\, dx\\*
&\leq\halb \int_{\partial(G\setminus\overline{\Om})}\!\!\!k^2\mu_2|u|^2 x\cdot n\, d\sigma+\Re\Bigl(\int_\Om k^2\mu_1 ux\cdot \nabla u^*\, dx+\int_{\partial G}ikux\cdot \nabla u^*\, d\sigma\Bigr)\\*
&\quad+\Re\Bigl(\int_G fx\cdot \nabla u^*\, dx+\int_{\partial G}gx\cdot \nabla u^*\, d\sigma\Bigr).
\end{align*}
The assumption \eqref{eq:geometry} on $G$ and $\Om$ implies that the first term on the right-hand side can be bounded above by $Ck^2\|u\|_{\partial G}^2$. This yields after application of H\"older and Young inequalities
\[k^2\|u\|^2_{G\setminus \overline{\Om}}\leq C(k^2\|u\|_\Om\|\nabla u\|_\Om+k^2\|u\|^2_{\partial G}+\|g\|^2_{\partial G}+\|f\|_G\|\nabla u\|_G).\]
Inserting the estimates \eqref{subeq:imucase3} and \eqref{subeq:realucase3} into the estimate for $k^2\|u\|^2_{G\setminus \overline{\Om}}$ gives
\begin{align*}
k^2\|u\|^2_{G\setminus \overline{\Om}}&\leq C\Bigl(\|g\|^2_{\partial G}\!+\!\frac{1}{kc_0}\|f\|_\Om^2\!+\eta_1k^2\|u\|_{G\setminus\overline{\Om}}^2\!+\!\frac{1}{\eta_1}\|f\|^2_{G\setminus\overline{\Om}}+\frac{1}{\eta_2}\|f\|_G^2+\eta_2k^2\|u\|^2_{G\setminus\overline{\Om}}\\*
&\qquad\quad+\frac{\eta_2}{k^2}(1+c_0^{-2})\|f\|^2_G+\frac{\eta_2}{k^2c_0}\|f\|^2_\Om+\frac{\eta_2}{k}(1+c_0^{-1})\|g\|^2_{\partial G}+\frac{k^4}{\de_2}\|u\|_\Om^2\Bigr).
\end{align*}
Choose $\eta_1,\eta_2$ independent of $k$ such that $k^2\|u\|_{G\setminus \overline{\Om}}$ can be hidden on the left-hand side and insert once more \eqref{subeq:imucase3} for the last term on the right-hand side to obtain
\begin{align*}
k^2\|u\|^2_{G\setminus \overline{\Om}}&\leq C\Bigl(\|g\|^2_{\partial G}+\|f\|^2_G+\Bigl(\frac{1}{kc_0}+\frac{1}{k^2c_0}\Bigr)\|f\|^2_\Om+\Bigl(\frac{1}{k^2}+\frac{1}{k^2c_0^2}\Bigr)\|f\|^2_G\\*
&\qquad \quad+\Bigl(\frac{1}{k}+\frac{1}{kc_0}\Bigr)\|g\|^2_{\partial G}+\frac{k}{c_0}\|g\|^2_{\partial G}+\frac{1}{c_0^2}\|f\|^2_\Om\\*
&\qquad\quad+\eta_3k^2\|u\|^2_{G\setminus \overline{\Om}}+\frac{k^2}{\eta_3 c_0^2}\|f\|^2_{G\setminus\overline{\Om}}\Bigr).
\end{align*}
Choosing finally $\eta_3$ appropriately gives the desired estimate for $k^2\|u\|^2_{G\setminus \overline{\Om}}$ and combination with \eqref{subeq:imucase3} and \eqref{subeq:realucase3} finishes the proof.
\end{proof} 

If $c_0$ is independent from $k$, we obtain
\begin{align*}
\|u\|_{1,k,G}&\leq C(\|f\|_\Om+k\|f\|_{G\setminus\overline{\Om}}+k^{1/2}\|g\|_{\partial G}).
\end{align*}
On the other hand, if $c_0>k^{-2}$ as in the case of $\mu_{\eff}$ (see Proposition \ref{prop:effective}), we obtain
\begin{align*}
\|u\|_{1,k,G}&\leq C(k^2\|f\|_\Om+k^3\|f\|_{G\setminus\overline{\Om}}+k^{3/2}\|g\|_{\partial G}).
\end{align*}
The dependence of $c_0$ on $k$ contributes by a factor $k$ for $g$ and a factor $k^2$ for $f$.
However, even without this critical dependence of $c_0$ on $k$, the stability estimate is worse than the classical versions of about a factor $k$ for $f$ and $k^{1/2}$ for $g$. Looking into the proof, one can see that this is due to the difficult term $\int_\Om k^2\mu ux\cdot \nabla u$.

The presented proof can also be transferred (with minor adaptations) to the case where $\mu$ is a real constant and then yields the known stability of $k^0$. So this also contributes to the analysis of \cite{BGP15hethelmholtzLOD} by covering the case of matrix-valued $a$.

\begin{proof}[Proof of Theorem \ref{thm:stabilityeff}]
Because of the density of smooth functions in $L^p$ for $p\in [1,\infty)$, for every $\eta>0$ there exists $a_\eta\in C^\infty(\overline{G})$ such that $\|a_\eta-a\|_{L^p}\leq \eta$. Furthermore, $a_\eta$ can be chosen symmetric and uniformly elliptic with constants independent of $\eta$. \
Because of the additional assumption on $a_{\eff}$ and the geometric setting, the assumption ``$Da_\eta\cdot x$ is negative semi-definite'' can also be fulfilled for all $\eta$ small enough.
In the sequel, $C$ is a generic constant, independent of $k$ and $\eta$.

The solution $u_\eta$ to the Helmholtz problem with diffusion coefficient $a_\eta$ (and sesquilinear form $B_\eta$) satisfies according to the previous proposition
\[\|u_\eta\|_{1,k,G}\leq C(k^3\|f\|_{G\setminus\overline{\Om}}+k^2\|f\|_\Om+k^{3/2}\|g\|_{\partial G}).\]
$u-u_\eta$ satisfies $B_\eta(u-u_\eta, v)=\int_G(a_\eta-a)\nabla u\cdot\nabla v^*$ for all $v\in H^1(G)$. As the inf-sup-constant of $B_\eta$ is bounded below by $k^{-4}$, this gives
\[\|u-u_\eta\|_{1,k,G}\leq Ck^4\|(a_\eta-a)\nabla u\|_G.\]

By the H\"older inequality, we have $\|(a_\eta-a)\nabla u\|_G\leq C\|a_\eta-a\|_{L^p}\|\nabla u\|_{L^q}$ for all $p,q$ with $1/p+1/q=1/2$. Now choose $q$ such that $L^q\subset H^s$ for some $s\in (0,1/2]$ (e.g.\ $q=p=4$ or $q=8/3$, $p=8$). 
Because of $\|a_\eta-a\|_{L^p}\leq \eta$ and the estimate for the $H^s$-norm of $u$ (see Proposition \ref{prop:regularity}), we get
\[\|u-u_\eta\|_{1,k,G}\leq Ck^4\eta(k\|u\|_{1,k,G}+\|f\|_G+\|g\|_{H^{1/2}(\partial G)}).\]
Now choose $\eta=O(k^{-5})$ small enough. By the triangle inequality we finally obtain
\begin{align*}
\|u\|_{1,k,G}&\leq \|u-u_\eta\|_{1,k,G}+\|u_\eta\|_{1,k,G}\\
&\leq \halb\|u\|_{1,k,G}+Ck^{-1}(\|f\|_G+\|g\|_{H^{1/2}(\partial G)})\\*
&\quad+C(k^3\|f\|_{G\setminus \overline{\Om}}+k^2\|f\|_\Om+k^{3/2}\|g\|_{\partial G}),
\end{align*}
which gives the claim.
\end{proof}

\subsection{Proof of the quasi-optimality of the HMM}
\label{subsec:proofquasiopt}
In this section we give the proof of our central result, namely Theorem \ref{thm:infsupdiscr}.

\begin{proof}[Proof of Theorem \ref{thm:infsupdiscr}]

{\itshape Proof of the discrete inf-sup constant \eqref{eq:discrinfsup}}:
Let $\Vv_{H,h}:=(v_H, v_{h,1}, v_{h,2})\in \VV_{H,h}$ be given and let $\Vz:=(z, z_1, z_2)\in \mathcal{H}$ solve
\begin{align*}
\CB(\Vpsi, \Vz) &= 2k^2\int_G\int_Y(\psi+\chi_D \psi_2)(v_H^*+\chi_D v_{h,2}^*)\, dy dx
\qquad\forall\Vpsi:=(\psi, \psi_1, \psi_2)\in \mathcal{H}.
\end{align*}
Due to the regularity of the cell problems (Proposition \ref{prop:regularity}), Assumption \ref{asspt:polynomialstable} on the stability, and the resulting estimates from Proposition \ref{prop:infsupconst} it holds
\begin{equation}
\label{eq:regz}
\begin{split}
\|\Vz\|_e&\leq 2C_{\stab, e} k^{q+1} \|\Vv_{H, h}\|_e,\\ 
\|z\|_{H^{1+s_0}_{pw}(G)}&\leq 2C_{\reg, 0} k^{q+2}\|\Vv_{H,h}\|_e,\\
\|z_1\|_{L^2(\Om; H^{1+s_1}(Y^*))}\leq C_{\reg, 1} \|\Vz\|_e&\leq 2C_{\reg, 1}C_{\stab, e}k^{q+1}\|\Vv_{H, h}\|_e,\\
\|z_2\|_{L^2(\Om; H^{1+s_2}(D))}\leq C_{\reg, 2}k\|\Vz\|_e&\leq 2C_{\reg,2}C_{\stab, e}k^{q+2}\|\Vv_{H, h}\|_e.
\end{split}
\end{equation} 
Due to \eqref{eq:approxvHh} we can choose $\Vz_{H,h}:=(z_H, z_{h,1}, z_{h,2})\in \VV_{H,h}$ such that
\begin{equation}
\label{eq:approxz}
\begin{split}
\|\Vz-\Vz_{H, h}\|_e&\leq C_{\appr}(H^{s_0}(1+kH)\|z\|_{H^{1+s_0}_{pw}(G)}+h^{s_1}\|z_1\|_{L^2(\Om; H^{1+s_1}(Y^*))}\\
&\qquad\qquad+h^{s_2}(1+kh)\|z_2\|_{L^2(\Om; H^{1+s_2}(D))})\\
&\stackrel{\eqref{eq:regz}}{\leq} 2C_{\appr}\bigl(C_{\reg, 0}k^{q+2}H^{s_0}(1+kH)+C_{\reg, 1}C_{\stab, e}k^{q+1}h^{s_1}\\
&\qquad\qquad\quad+C_{\reg, 2}C_{\stab, e}k^{q+2}h^{s_2}(1+kh)\bigr)\|\Vv_{H, h}\|_e.
\end{split}
\end{equation}
With this $\Vz_{H, h}$ we obtain
\begin{align*}
\Re\CB(\Vv_{H, h}, \Vv_{H, h}+\Vz_{H, h})&=\Re\CB(\Vv_{H, h}, \Vv_{H, h}+\Vz-\Vz+\Vz_{H, h})\\*
&= \Re\CB(\Vv_{H, h}, \Vv_{H, h}+\Vz)-\Re\CB(\Vv_{H, h}, \Vz-\Vz_{H, h})\\*
&\geq C_{\min}\|\Vv_{H, h}\|_e^2-C_B\|\Vv_{H, h}\|_e\, \|\Vz-\Vz_{H, h}\|_e.
\end{align*}
Inserting \eqref{eq:approxz}, we obtain
\begin{align*}
&\!\!\!\!\Re\CB(\Vv_{H, h}, \Vv_{H, h}+\Vz_{H, h})\\
&\geq C_{\min}\Bigl(1-\frac{2C_BC_{\appr}}{C_{\min}}(C_{\reg, 0}k^{q+2}H^{s_0}(1+kH)+C_{\reg, 2}C_{\stab,e}k^{q+2}h^{s_2}(1+kh)\\*
&\hspace{4cm}+C_{\reg, 1}C_{\stab, e}k^{q+1}h^{s_1})\Bigr)\|\Vv_{H, h}\|_e^2.
\end{align*}
Hence, under the resolution conditions \eqref{eq:resolutionasspt}, this gives $\Re\CB(\Vv_{H, h}, \Vv_{H, h}+\Vz_{H, h})\geq \halb C_{\min}\|\Vv_{H, h}\|_e^2$. Finally, observing that
\begin{align*}
&\!\!\!\!\|\Vv_{H, h}+\Vz_{H, h}\|_e\leq \|\Vv_{H, h}\|_e+\|\Vz\|_e+\|\Vz-\Vz_{H, h}\|_e\\*
&\leq \bigl(1+2C_{\stab, e}k^{q+1}+2C_{\appr}(C_{\reg, 0}k^{q+2}H^{s_0}(1+kH)+C_{\reg, 1}C_{\stab, e}k^{q+1}h^{s_1}\\*
&\hspace{5cm}+C_{\reg, 2}C_{\stab, e}k^{q+2}h^{s_2}(1+kh))\bigr)\|\Vv_{H, h}\|_e\\
&\stackrel{\eqref{eq:resolutionasspt}}{\leq}\Bigl(1+2C_{\stab, e}k^{q+1}+\frac{C_{\min}}{2C_B}\Bigr)\|\Vv_{H,h}\|_e\\
&\leq\Bigl(k_0^{-(q+1)}\Bigl(1+\frac{C_{\min}}{2C_B}\Bigr)+2C_{\stab, e}\Bigr)k^{q+1}\|\Vv_{H,h}\|_e
\end{align*}
finishes the proof of the inf-sup condition.

\smallskip
\noindent
{\itshape Proof of the quasi-optimality \eqref{eq:hmmapriori}:}
Consider the following (auxiliary) dual problem for $\Vz:=(z, z_1, z_2)\in \mathcal{H}$
\begin{align*}
\CB(\Vpsi, \Vz) &= k^2\int_G\int_Y(\psi+\chi_D\psi_2)(e_0^*+\chi_D e_2^*)\, dydx
 \qquad \forall\Vpsi:=(\psi, \psi_1, \psi_2)\in \mathcal{H}.
\end{align*}
As already argued in the proof of the discrete inf-sup constant, $z \in H^{1+s_0}_{pw}(G)$ fulfills the estimate $\|z\|_{H^{1+s_0}_{pw}}\leq  C_{\reg, 0}k^{q+2}\|(e_0, e_1, e_2)\|_e$ due to Proposition \ref{prop:infsupconst}. For all $\Vz_{H,h}\in \VV_{H,h}$, the standard Galerkin orthogonality gives
\[k^2\|e_0+\chi_D e_2\|^2_{L^2(G\times Y)}=\CB(\Ve, \Vz)=\CB(\Ve, \Vz-\Vz_{H,h}).\]
The continuity of $\CB$ w.r.t.\ the energy norm and an approximation argument like \eqref{eq:approxz} yield
\begin{align*}
k^2\|e+\chi_D e_2\|^2_{L^2(G\times Y)}
&\leq C_B\|(e_0, e_1, e_2)\|_e\, \|\Vz-\Vz_{H,h}\|_e\\
&\leq C_BC_{\appr}\bigl(C_{\reg, 0}k^{q+2}H^{s_0}(1+kH)+C_{\reg, 1}C_{\stab, e}k^{q+1}h^{s_1}\\*
&\qquad\qquad\qquad+C_{\reg, 2}C_{\stab, e}k^{q+2}h^{s_2}(1+kh)\bigr)\|(e_0, e_1, e_2)\|^2_e.
\end{align*}
With the G{\aa}rding inequality, we get for any $\Vz_{H,h}\in \VV_{H,h}$
\begin{align*}
\|(e_0, e_1, e_2)\|_e ^2&\leq C_{\min}^{-1}\bigl(\Re\CB(\Ve, \Ve) +2k^2\|e_0+\chi_D e_2\|^2_{L^2(G\times Y)}\bigr)\\
&=\Re\bigl(\CB(\Ve, \Vu- \Vz_{H, h})+2k^2\|e_0+\chi_D e_2\|^2_{L^2(G\times Y)}\bigr)\\
&\leq \frac{C_B}{C_{\min}}\|\Vu-\Vz_{H, h}\|_e\, \|(e_0, e_1, e_2)\|_e\\*
&\quad+ \frac{2C_BC_{\appr}}{C_{\min}}\bigl(C_{\reg, 0}k^{q+2}H^{s_0}(1+kH)+C_{\reg, 1}C_{\stab, e}k^{q+1}h^{s_1}\\*
&\hspace{2.5cm}+C_{\reg, 2}C_{\stab, e}k^{q+2}h^{s_2}(1+kh)\bigr) \|(e_0, e_1, e_2)\|_e ^2.
\end{align*}
Together with the resolution conditions \eqref{eq:resolutionasspt} this gives
\[\|(e_0, e_1, e_2)\|_e^2\leq \frac{C_B}{C_{\min}}\|\Vu-\Vz_{H,h}\|_e\|(e_0, e_1, e_2)\|+\halb \|(e_0, e_1, e_2)\|_e^2\]
and hence the first inequality of \eqref{eq:hmmapriori}. The second inequality directly follows from the approximation properties \eqref{eq:approxvHh} and the regularity estimates from  Propositions \ref{prop:regularity} and \ref{prop:infsupconst}.
\end{proof}

\section{Numerical experiment}
\label{sec:numexperiment}
In this section we analyze the HMM numerically with particular respect to the convergence order (see Theorem \ref{thm:infsupdiscr}), the resolution condition (see \eqref{eq:resolutionasspt}) and the behavior of solutions for different wavenumbers $k$ and different values of $\mu_{\eff}$. 
The implementation was done with the module \textsf{dune-gdt} \cite{wwwdunegdt} of the DUNE software framework \cite{MR2421580,MR2421579}.

We consider the macroscopic domain $G=(0.25, 0.75)^2$ with embedded scatterer \linebreak[4]$\Omega=(0.375, 0.625)^2$. 
The boundary condition $g$ is computed as $g=\nabla u_{inc}\cdot n-iku_{inc}$ from the (left-going) incoming plane wave $u_{inc}=\exp(-ikx_1)$.
The unit square $Y$ has the inclusion $D=(0.25, 0.75)^2$ and the inverse permittivities are given as $\varepsilon_e^{-1}=10$ and $\varepsilon_i^{-1}=10-0.01i$.
Obviously, the real parts of both parameters are of the same order, and, moreover, $\varepsilon_i$ is only slightly dissipative.

\begin{figure}
\begin{center}
\begin{tikzpicture}
\begin{axis}[axis equal image=false, legend entries={$\Re(\mu_{\eff})$, $\Im(\mu_{\eff})$}, xmin=15, xmax=68, ymin=-25, ymax=25, legend style={at={(0.5, 1.0)}, anchor=north}]
\addplot+[no markers, densely dotted, thick] table[x=k, y= Re(mu), col sep=comma]{includes/mueff_ima001.csv};
\addplot+[no markers, densely dashed, thick] table[x=k, y= Im(mu), col sep=comma]{includes/mueff_ima001.csv};
\end{axis}
\end{tikzpicture}
\end{center}
\caption{Real and imaginary part of $\mu_{\eff}$ for changing wavenumber $k$.}
\label{fig:mueff}
\end{figure}
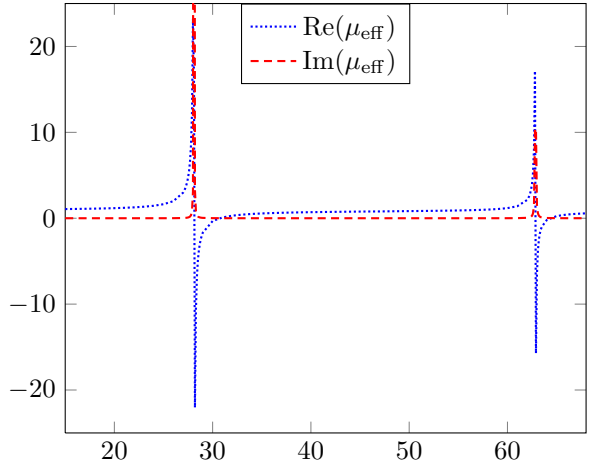

As the inclusion $D$ is quadratic, the eigenvalues of the Dirichlet Laplacian are explicitly known.
Only the eigenvalues where the associated eigenfunctions have non-zero mean contribute to the expansion of $\mu_{\eff}$.
For our setup, the first interesting values are at $k\approx 28.1$ and $k\approx 62.8$.
We compute $\mu_{\eff}$ using cell problem \eqref{eq:cellmu} with a grid consisting of $32768$ elements on $D$.
Figure \ref{fig:mueff} shows the behavior of the real and the imaginary part. 
As predicted, we can see a significant change of behavior around the Laplace eigenvalues, where the real part changes sign and also the imaginary part has large values.
Note that for this example, we do not see a dependence of $\Im(\mu_{\eff})$ like $k^{-2}$, as proved in Proposition \ref{prop:effective}.

\begin{figure}
\begin{center}
\begin{tikzpicture}
\begin{loglogaxis}[axis equal image=false, legend entries={$k=34$, $k=48$, $k=68$}]
\addplot table[x=num_entities, y=1khom, col sep=comma, row sep=\\]{
num_entities, 1khom\\
128, 15.7872178524\\
288, 9.8244204247\\
512, 6.8234878455\\
1152, 3.9859883145\\
2048, 2.7581152361\\
4608, 1.6847272363\\
8192, 1.2137186425\\
18432, 0.790574319\\
};
\addplot table[x=num_entities, y=1khom, col sep=comma, row sep=\\]{
num_entities, 1khom\\
128, 29.6553586362\\
288, 24.3573872912\\
512, 18.64780506\\
1152, 11.2642997948\\
2048, 7.5231704867\\
4608, 4.1700970319\\
8192, 2.8196776464\\
18432, 1.6879761326\\
};
\addplot+[mark=diamond*] table[x=num_entities, y=1khom, col sep=comma, row sep=\\]{
num_entities, 1khom\\
128, 44.4068489542\\
512, 42.7011355978\\
1152, 37.6854385529\\
2048, 25.9186036203\\
4608, 13.9224194377\\
8192, 8.765414034\\
18432, 4.5290776394\\
};
\end{loglogaxis}
\end{tikzpicture}
\end{center}
\caption{Error between homogenized reference solution and macroscopic part $u_H$ of the HMM approximation in weighted $H^1$-norm vs.\ number of grid entities for different wavenumbers $k$.}
\label{fig:resolcond}
\end{figure}
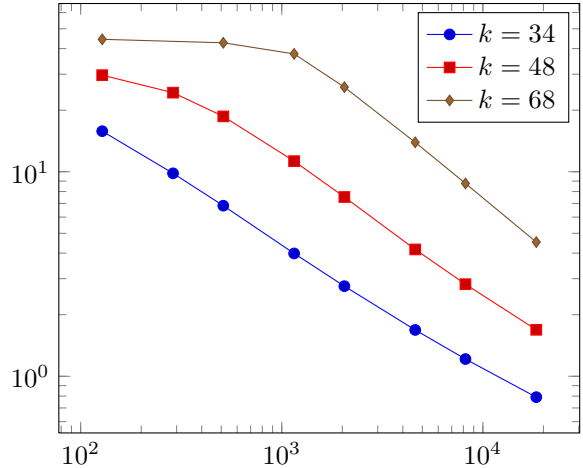

In order to analyze the resolution condition, we use a reference homogenized solution by computing the effective parameters with $524288$ entities on $Y$ and then solving the effective homogenized equation on $G$ with the same number of entities.
We compare the macroscopic part $u_H$ of our HMM approximation with this reference solution in the weighted $H^1$-norm $\|\cdot\|_{1,k,G}$ for a sequence of simultaneously refined macro- and finescale meshes and three different wavenumbers $k=34$, $k=48$, $k=68$, see Figure \ref{fig:resolcond}.
Note that these wavenumbers are all away from any resonant behavior of $\mu_{\eff}$.
For higher wavenumbers, finer meshes are needed to obtain convergence: Whereas for $k=34$, the error convergences for all considered grids, the threshold value for $k=68\approx \sqrt{2}\times 34$ is $288$ entities; and for $k=68=2\times 34$, it is $1152$ entities.
This indicates a resolution condition of ``$k^2(H+h)$ small'' in practice, which is standard for continuous Galerkin discretizations of Helmholtz problems.

\begin{table}
\caption{Convergence history and EOC for the error between the macroscopic part $u_H$ of the HMM approximation and the reference homogenized solution in $L^2$-norm and $k$-weighted $H^1$-norm.}
\label{tab:errorhomk29}
\centering
\begin{tabular}{@{}ccccc@{}}
\toprule
$H=2h$&$\|e_0\|_{L^2(G)}$&$\|e_0\|_{1, k, G}$&EOC($\|e_0\|_{L^2}$)&EOC($\|e_0\|_{1,k}$)\\
\midrule
$\sqrt{2}\times 1/8$ & $0.270474$ & $11.7804630632$ & ---& ---\\
$\sqrt{2}\times 1/12$ & $0.197617$ & $8.9454269415$ & $0.7740374081$ & $0.678973445$\\
$\sqrt{2}\times 1/16$ & $0.110372$ & $5.373206314$ & $2.0247154456$ & $1.7718088298$\\
$\sqrt{2}\times 1/24$ & $0.0513966$ & $2.9702496635$ & $1.2792537865$ & $1.4619724025$\\
$\sqrt{2}\times 1/32$ & $0.0296714$ & $2.0192725797$ & $1.9097067775$ & $1.3414415096$\\
$\sqrt{2}\times 1/48$ & $0.0135056$ & $1.2358350102$ & $1.9411761676$ & $1.2109315066$\\
$\sqrt{2}\times 1/64$ & $0.00767201$ & $0.8863106904$ & $1.9658012347$ & $1.1555624022$\\
\bottomrule
\end{tabular}
\end{table}
\begin{table}
\caption{$L^2(G)$-norm of the error to the reference heterogeneous solution for macroscopic part $u_H$ and zeroth order reconstruction $u_{\HMM}^0$.}
\label{tab:errorcorreck29}
\centering
\begin{tabular}{@{}cccc@{}}
\toprule
$H=2h$&$\|u_\delta-u_H\|_{L^2(G)}$&$\|u_\delta-u_{\HMM}^0\|_{L^2(G)}$&EOC($u_\delta-u_{\HMM}^0$)\\
\midrule
$\sqrt{2}\times 1/8$ & $0.418463$ & $0.565853$ & ---\\
$\sqrt{2}\times 1/16$ & $0.351655$ & $0.174724$ & $1.695349522$\\
$\sqrt{2}\times 1/24$ & $0.34595$ & $0.0619639$ & $2.5567073663$\\
$\sqrt{2}\times 1/32$ & $0.346266$ & $0.0340908$ & $2.0770303799$\\
$\sqrt{2}\times 1/48$ & $0.34733$ & $0.0272449$ & $0.5528495573$\\
$\sqrt{2}\times 1/64$ & $0.347862$ & $0.0297642$ & $-0.3074226373$\\
\bottomrule
\end{tabular}
\end{table}

We now take a closer look at the convergence of the errors and verify the predictions of Theorem \ref{thm:infsupdiscr}.
We choose the wavenumber $k=29$, which corresponds to $\Re(\mu_{\eff})<0$ and thus is also interesting from a physical point of view.
Table \ref{tab:errorhomk29} shows the error between the macroscopic part $u_H$ of the HMM approximation and the reference homogenized solution (as before) in the $k$-weighted $H^1(G)$-norm and the $L^2(G)$-norm.
The experimental order of convergence (EOC), defined as  EOC($e):=\ln(\frac{e_{H_1}}{e_{H_2}})/\ln(\frac{H_1}{H_2})$, verifies the linear convergence in the $H^1$-norm predicted theoretically in Theorem \ref{thm:infsupdiscr}, and the quadratic convergence in the $L^2$-norm discussed afterwards.
This clearly shows that our general theory holds for all regimes of wavenumbers even if they result in unusual effective parameters.
However, we observe a small pre-asymptotic effect for coarse meshes, which  indicates that the resolution condition may be stricter for those resonant settings.
Furthermore, we compare the HMM approximation with a detailed reference solution of the heterogeneous problem for $\delta=1/32$, solved on a fine grid with $524288$ entities.
Table \ref{tab:errorcorreck29} compares the error to the reference solution for the macroscopic part $u_H$ of the HMM approximation and  for the zeroth order $L^2$-approximation $u_{\HMM}^0=u_H+\delta u_{h,2}(\cdot, \frac{\cdot}{\delta})$.
Whereas the error stagnates for $u_H$, we almost recover the quadratic convergence for $u_{\HMM}^0$ with a saturation effect for fine meshes where we enter the regime of the homogenization error.
This clearly underlines the necessity of the correctors in the HMM to faithfully approximate the true solution.
Note that we do not have results on the homogenization error: We expect strong convergence of $u_\delta$ to $u_{\HMM}^0$ in the $L^2$-norm according to \cite{All92twosc}, but the proof is not applicable to the Helmholtz case.

\begin{figure}
\centering
\subfloat[$u_H$]{\includegraphics[width=0.47\textwidth, trim= 31mm 53mm 1mm 53mm, clip=true]{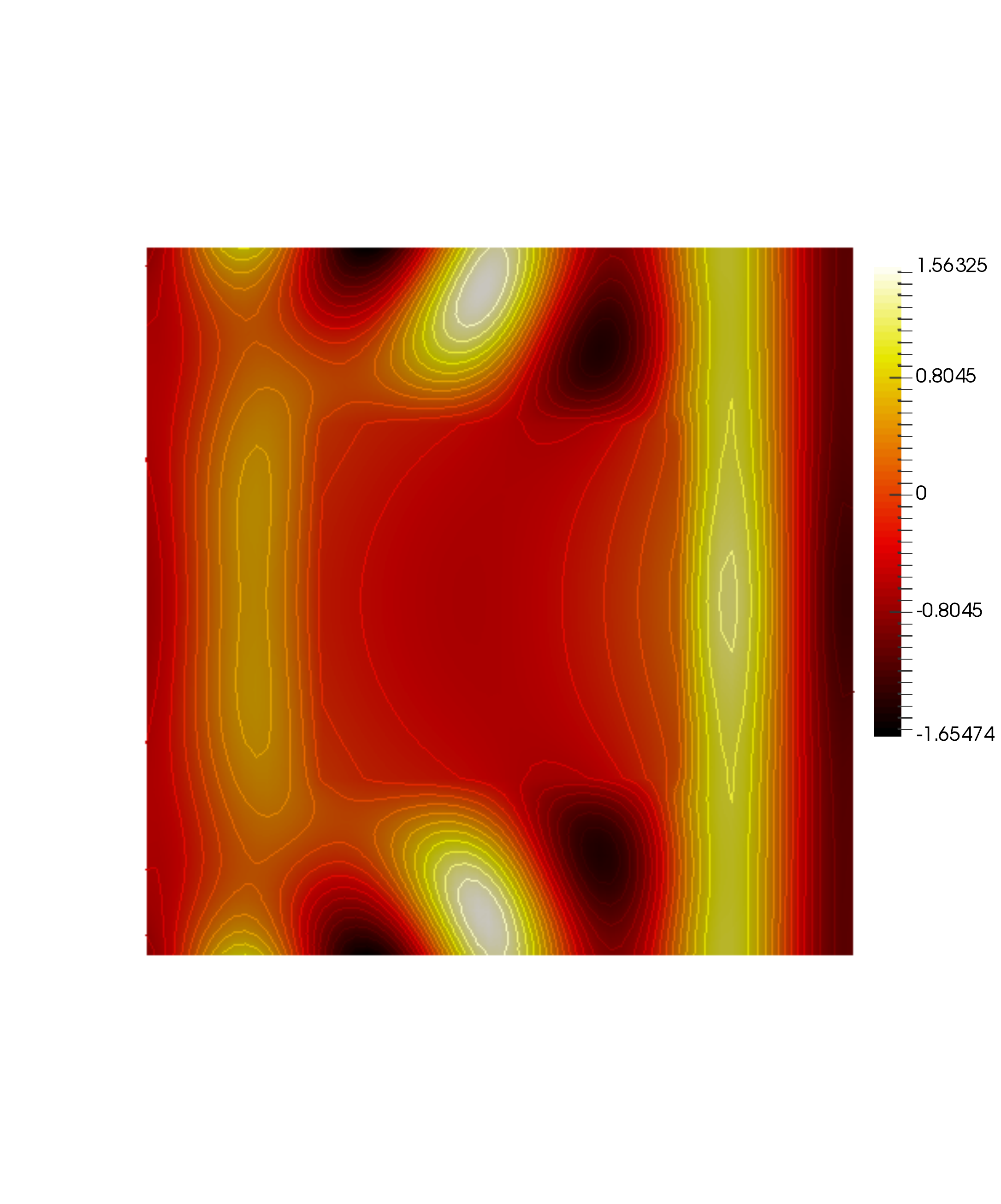}
\label{subfig:homk38}
}%
\hspace{10pt}%
\subfloat[line plot of $u_H$]{\includegraphics[width=0.43\textwidth]{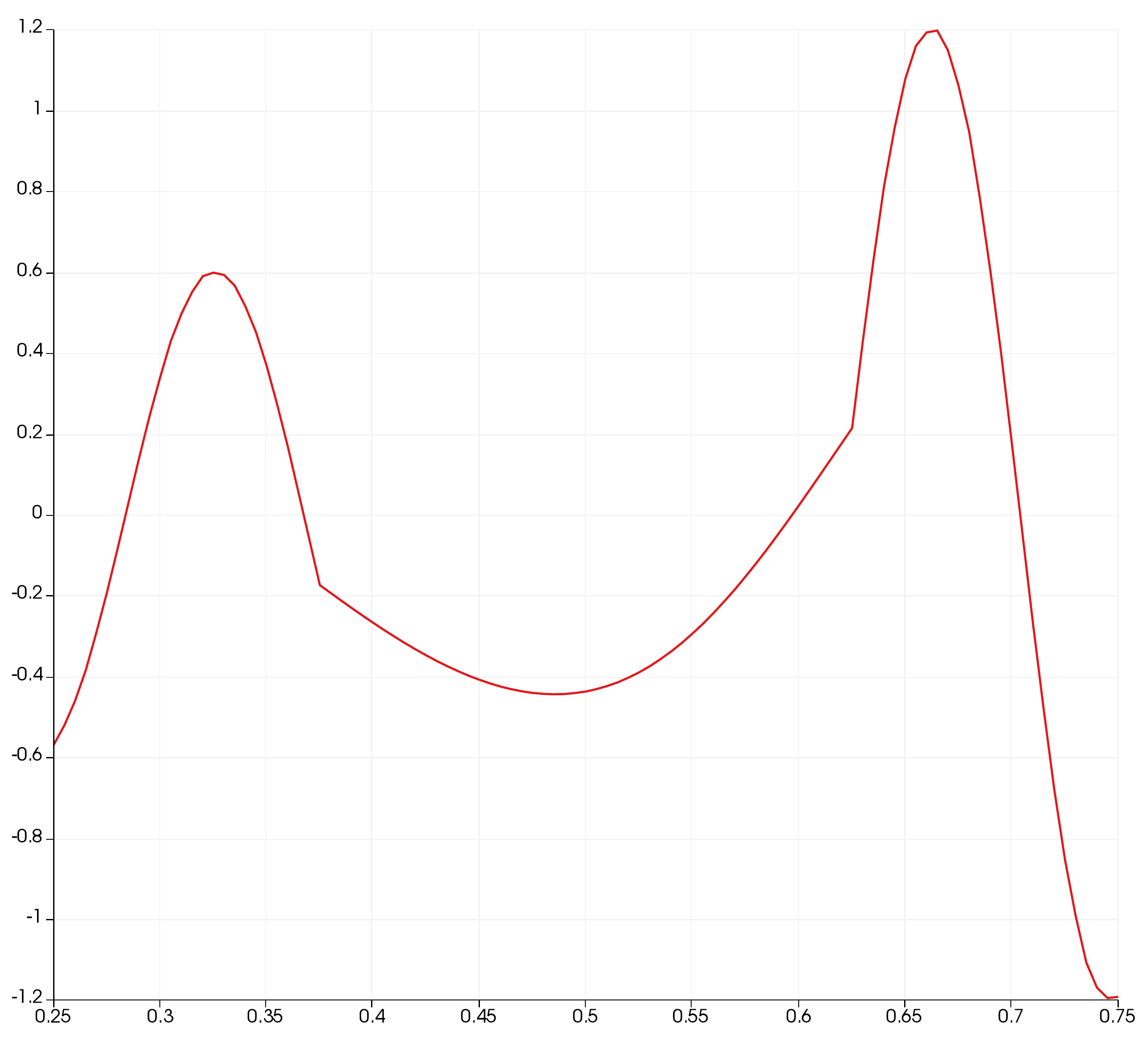}
\label{subfig:homk38line}
}
\\
\subfloat[$u_{\HMM}^0$]{\includegraphics[width=0.48\textwidth, trim= 31mm 53mm 1mm 53mm, clip=true]{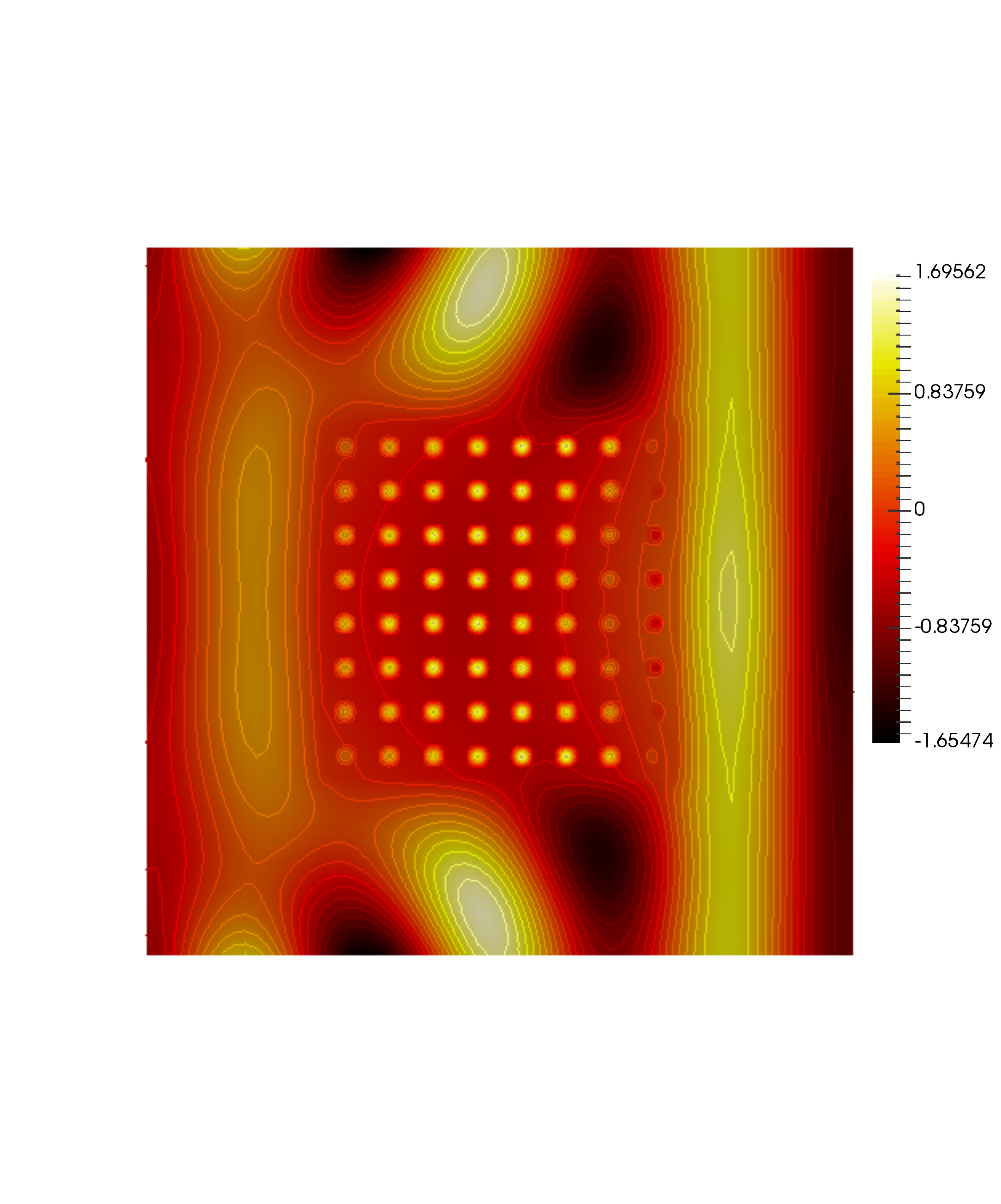}
\label{subfig:correck38}
}%
\hspace{10pt}%
\subfloat[line plot of $u_{\HMM}^0$]{\includegraphics[width=0.42\textwidth]{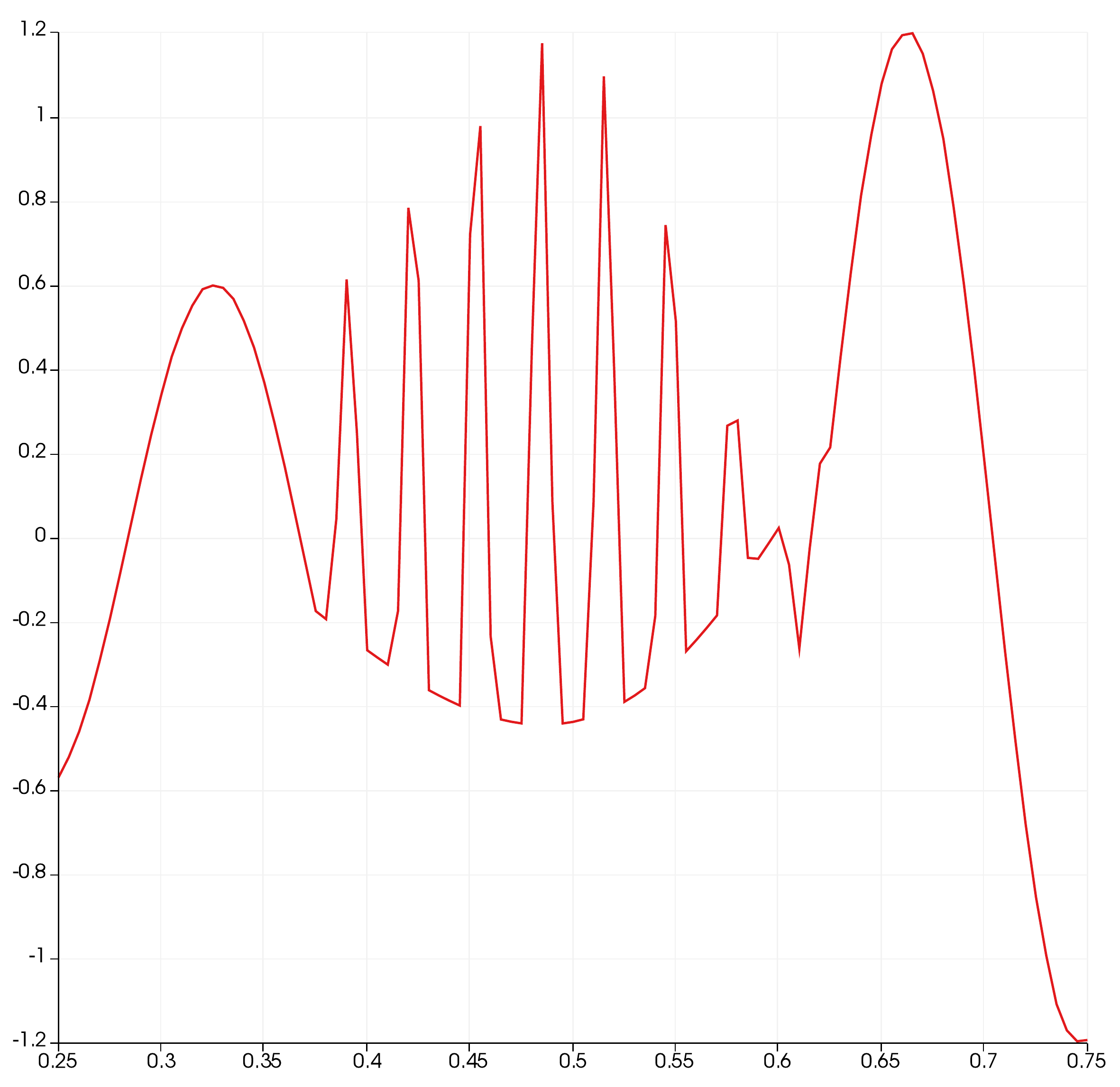}
\label{subfig:correck38line}
}
\caption{For $k=38$: Real part of the macroscopic part $u_H$ and real part of zeroth order reconstruction $u_{\HMM}^0$, both on the whole domain (left column) and over the line $y=0.545$ (right column). Computed with $H=2h=\sqrt{2}\times 1/64$; $u_H$ visualized on that grid, $u_{\HMM}^0$ on fine reference mesh.}
\end{figure}

\begin{figure}
\centering
\subfloat[$u_H$]{\includegraphics[width=0.53\textwidth, trim= 32mm 53mm 0mm 52mm, clip=true]{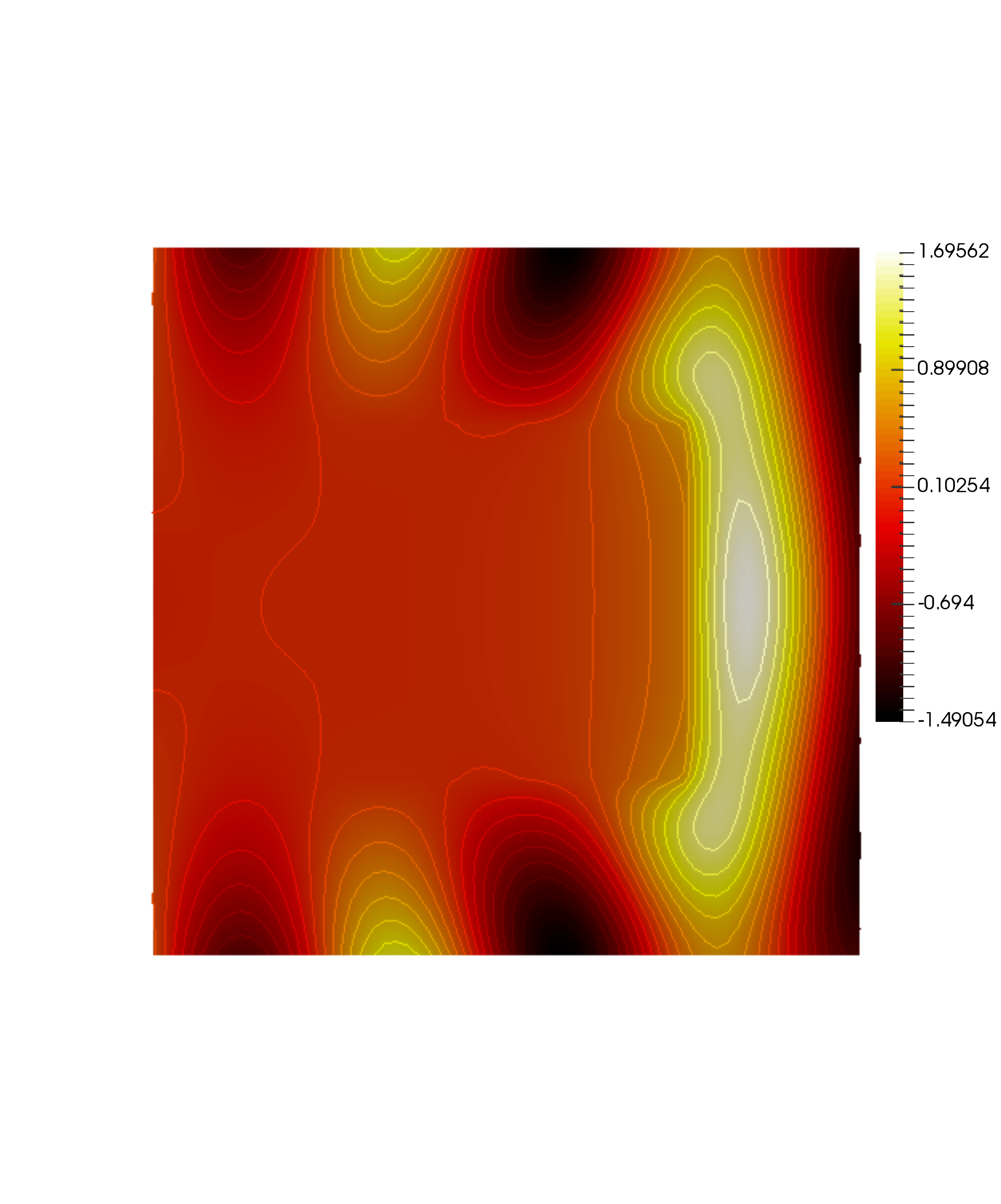}
\label{subfig:homk29}
}%
\hspace{10pt}%
\subfloat[line plot of $u_H$]{\includegraphics[width=0.37\textwidth]{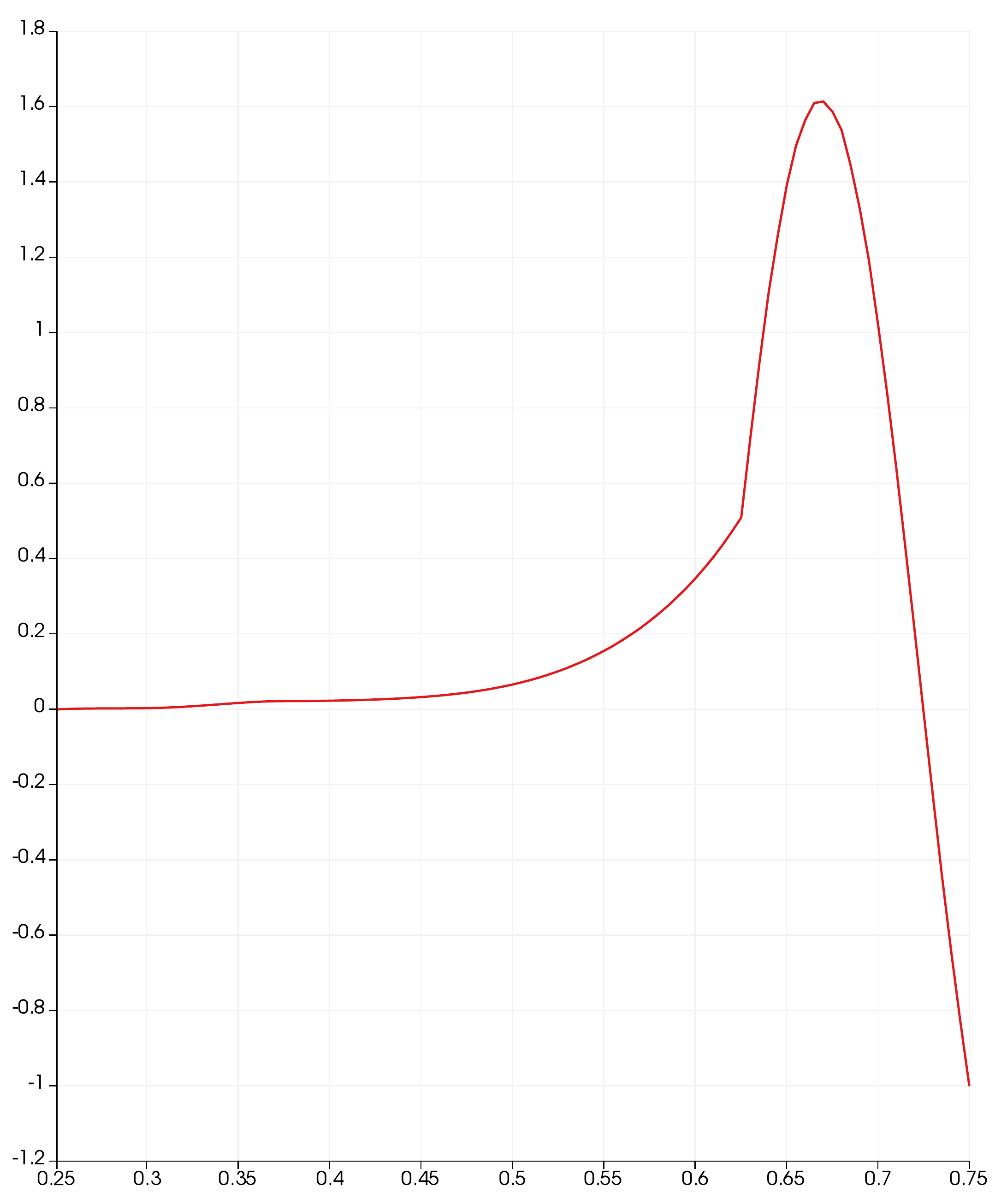}
\label{subfig:homk29line}
}
\\
\subfloat[$u_{\HMM}^0$]{\includegraphics[width=0.53\textwidth, trim= 30mm 72mm 2mm 72mm, clip=true]{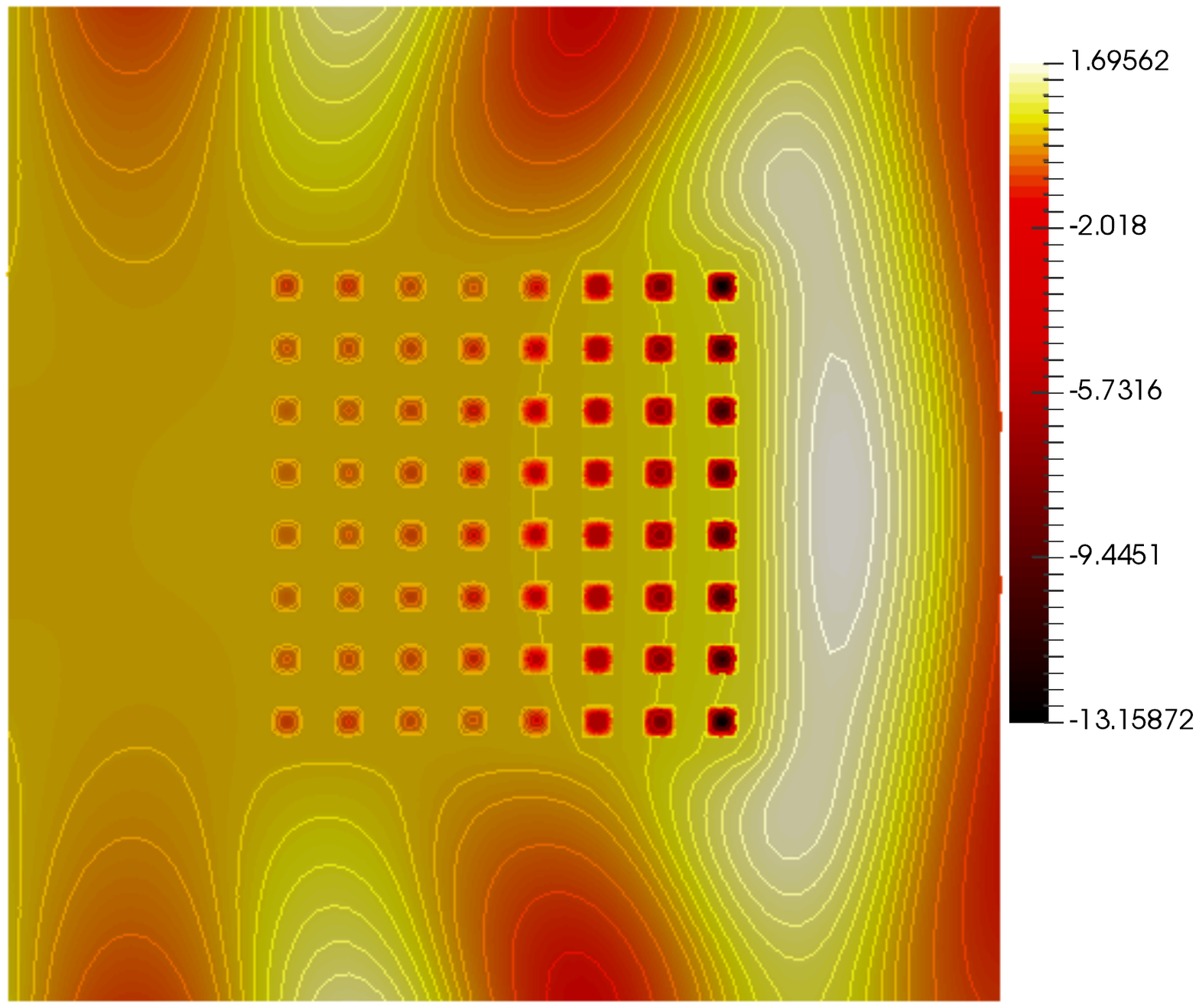}
\label{subfig:correck29}
}%
\hspace{10pt}%
\subfloat[line plot of $u_{\HMM}^0$]{\includegraphics[width=0.37\textwidth]{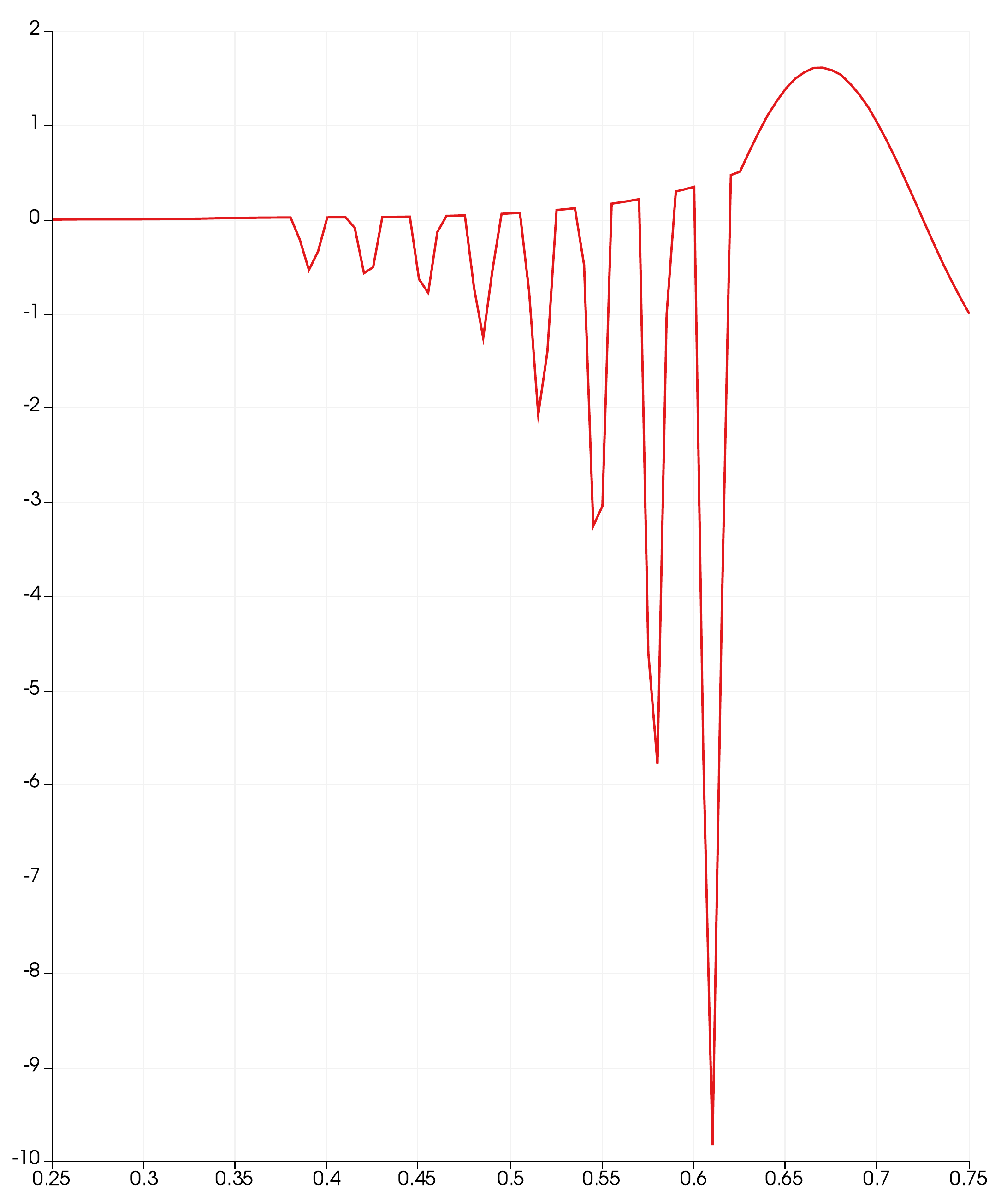}
\label{subfig:correck29line}
}
\caption{For $k=29$: Real part of the macroscopic part $u_H$ and real part of zeroth order reconstruction $u_{\HMM}^0$, both on the whole domain (left column) and over the line $y=0.545$ (right column). Computed with $H=2h=\sqrt{2}\times 1/64$; $u_H$ visualized on that grid, $u_{\HMM}^0$ on fine reference mesh.}
\end{figure}

Finally, we compare two wavenumbers with very different physical meaning: $k=38$ corresponds to normal transmission, whereas $k=29$ has $\Re(\mu_{\eff})<0$ and thus corresponds to a wavenumber in the band gap where propagation inside the scatterer is forbidden.
We consider the macroscopic part $u_H$ of the HMM approximation (with $H=2h=\sqrt{2}\times 1/64$) and the zeroth order reconstruction $u_{\HMM}^0$ (plotted on a well resolved mesh with $524288$ entities) and depict both functions on the whole two-dimensional domain as well as over the line $y=0.545$, which cuts through a row of inclusions.
For $k=38$, wave propagation with low speed takes place inside the scatterer, see the macroscopic part $u_H$ depicted in Figure \ref{subfig:homk38} and \ref{subfig:homk38line}.
In contrast to that, we see the expected exponential decay of the wave inside the scatterer for $k=29$, see the macroscopic part $u_H$ depicted in Figure \ref{subfig:homk29} and \ref{subfig:homk29line}.
The zeroth order reconstruction $u_{\HMM}^0$ can explain this behavior by approximating the heterogeneous solution also inside the inclusion.
For $k=38$, the amplitudes inside the inclusions are as high as the amplitude of the incoming wave, see Figure \ref{subfig:correck38} and \ref{subfig:correck38line}.
However, we observe very high amplitudes inside the inclusions for $k=29$, see Figures \ref{subfig:correck29} and \ref{subfig:correck29line}.
These are caused by eigen resonances incited inside the inclusions.
Moreover, these incited waves from neighboring inclusions interfere destructively with each other so that over the whole scatterer, no wave can propagate.

\section*{Conclusion}
We suggested a new Heterogeneous Multiscale Method (HMM) for the Helmholtz equation with high contrast.  
The stability and regularity of the associated analytical two-scale solution is rigorously analyzed and thereby, a new stability estimate for Helmholtz equations with piecewise constant coefficients is developed.
The HMM is defined as direct finite element discretization of the two-scale equation, which is crucial for the numerical analysis. 
Quasi-optimality of the HMM under the (unavoidable) resolution condition ``$k^{q+2}(H+h)$ is sufficiently small'' is proved, where $q$ denotes the exponent for $k$ in the stability estimate. 
Numerical experiments verify the developed convergence results and analyze the resolution condition.
Moreover, the approximation to the heterogeneous solution, obtained from the HMM,  explains the effect of evanescent waves in frequency band gaps as destructive interference of eigen resonant waves inside the inclusions.

\section*{Acknowledgement}
The authors would like to thank P.\ Henning, A.\ Lamacz, and B.\ Schweizer for fruitful discussions on the subject.

\end{document}